\documentclass[a4paper]{article}
\usepackage[english]{babel}
\usepackage[utf8]{inputenc}
\usepackage{amsmath,amsthm}
\usepackage{amsfonts,amssymb}

\usepackage{enumitem}

\usepackage{mathtools}

\usepackage[colorlinks=true,linktocpage=true,linkcolor=black,citecolor=black]{hyperref}
\usepackage[capitalise]{cleveref} 

\usepackage{tikz,tikz-cd} 

  \DeclareMathSymbol{:}{\mathpunct}{operators}{"3A}

 \newcommand{\cof}{\text{cof}}
 \newcommand{\fib}{\text{fib}}

  \newcommand{\sievec}{\operatorname{\downarrow}}
  \newcommand{\sievei}{\mathrel{\triangleleft}}

   \tikzset{ito/.style={hook}}
   
   \tikzset{ito/.style={-triangle}}

   \newcommand{\cto}{\rightarrowtail}
   \tikzset{cto/.style={>->}}
   
   \newcommand{\acyc}{\sim}

   \newcommand{\triv}{\approx}

   \newcommand{\ano}{a}

   \tikzset{tcto/.style={cto,"\triv"}}
   
   \tikzset{acto/.style={cto,"\acyc"}}

   \newcommand{\fto}{\twoheadrightarrow}
   \tikzset{fto/.style={->>}}

   \tikzset{tfto/.style={fto,"\triv"}}
   
   \tikzset{afto/.style={fto,"\acyc"}}

   \renewcommand{\Pr}{\operatorname{Pr}}

\newcommand{\Wb}{\mathbb{W}}

\newcommand{\Acal}{\mathcal{A}} 
\newcommand{\Zcal}{\mathcal{Z}} 
\newcommand{\Ecal}{\mathcal{E}}

\newcommand{\Ycal}{\mathcal{Y}}

\newcommand{\Pcal}{\mathcal{P}} 
 
\newcommand{\Scal}{\mathcal{S}}

\newcommand{\Wcal}{\mathcal{W}} 
\newcommand{\Xcal}{\mathcal{X}} 
\newcommand{\Ccal}{\mathcal{C}} 
 
\newcommand{\Bcal}{\mathcal{B}}

\swapnumbers

\theoremstyle{plain}
\newtheorem{theorem}{Theorem}[section]
\newtheorem{lemma}[theorem]{Lemma}
\newtheorem{prop}[theorem]{Proposition}
\newtheorem{cor}[theorem]{Corollary}

\theoremstyle{remark}

\theoremstyle{definition}
\newtheorem{definition}[theorem]{Definition}
\newtheorem{remark}[theorem]{Remark}

\newtheorem{notation}[theorem]{Notation}

\newtheorem{example}[theorem]{Example}
\newtheorem{construction}[theorem]{Construction}
\newtheorem{assumptions}[theorem]{Assumptions}

\DeclareMathOperator*{\colim}{Colim}

\begin{document}

\pagestyle{plain}
\title{Minimal model structures}
\date{}

\author{Simon Henry}

\maketitle

\begin{abstract} We prove, without set theoretic assumptions, that every locally presentable category $\Ccal$ endowed with a tractable cofibrantly generated class of cofibrations has a unique minimal (or left induced) Quillen model structure. More generally, for any set $S$ of arrows in $\Ccal$ we construct the minimal model structure on $\Ccal$ with the prescribed cofibrations and making all the arrows of $S$ weak equivalences. We describe its class of equivalences as the ``smallest Cisinski localizer containing $S$''. Our proof rely on a careful use of the fat small object argument and J.~Lurie's ``good colimits'' technology and on the author previous work on combinatorial weak model categories and semi-model categories. We also obtain similar results for left semi-model categories.
\end{abstract}

\renewcommand{\thefootnote}{\fnsymbol{footnote}} 
\footnotetext{\emph{email:} shenry2@uottawa.ca}
\footnotetext{This work was partially supported by an NSERC Discovery Grant.}
\renewcommand{\thefootnote}{\arabic{footnote}} 



\tableofcontents

\section{Main results}

We call \emph{structured category}, a category $\Ccal$ with all limits and colimits and equipped with a weak factorization system (cofibrations, anodyne fibrations). 

\begin{definition}\label{def:localizer} A \emph{localizer} $\Wcal$ on a structured category $\Ccal$, or \emph{$\Ccal$-localizer}, is a class of arrows such that:

\begin{enumerate}[label=(\arabic*)]

\item $\Wcal$ satisfies the $2$-out-of-$3$ condition.

\item All anodyne fibrations of $\Ccal$ are in $\Wcal$.

\item The class of cofibrations that are in $\Wcal$ is closed under pushout and transfinite compositions. 

\end{enumerate}

\end{definition}

If $\Ccal$ is a Quillen model category, then its class of weak equivalences is clearly a localizer. Conversely, J.~Smith famously proved\footnote{Smith never published his proof. It originally appeared in print in \cite{beke2000sheafifiable} as Theorem~1.7. See also Proposition A.2.6.10 in \cite{lurie2009higher} and Proposition~1.7 in \cite{barwick2007enriched}. The last reference is the closest to what we claim here.} that if $\Ccal$ is a structured category, and $\Wcal$ is a $\Ccal$-localizer such that:

\begin{itemize}

\item $\Ccal$ is locally presentable,

\item the (cofibration,anodyne fibration) weak factorization system is cofibrantly generated,

\item $\Wcal$ satisfies the solution set condition, or equivalently is an accessible and accessibly embedded full subcategory of $\Ccal^\to$,

\end{itemize}

\noindent then $\Ccal$ has a combinatorial Quillen model structure whose class of equivalences is $\Wcal$ and whose cofibrations and anodyne fibrations are as specified above. 

We say that a structured category $\Ccal$ is \emph{combinatorial} when it satisfies the first two conditions above, i.e. $\Ccal$ is locally presentable and its weak factorization system is cofibrantly generated. We say, more precisely, that $\Ccal$ is \emph{$\lambda$-combinatorial} if it is locally $\lambda$-presentable and it admits a set of generating cofibrations between $\lambda$-presentable objects. We also recall the terminology

\begin{definition} A combinatorial structured category is said to be \emph{tractable} if it admits a set of generating cofibrations that have cofibrant domains. A model category\footnote{Or a pre-model category as in item \ref{item:def_premodel} of \cref{sec:wms}.} is said to be \emph{tractable} if it admits generating sets of cofibrations and anodyne cofibrations that have cofibrant domains. \end{definition}

In the present paper we aim at replacing the last condition of Smith's theorem by another one, namely the notion of small-generated localizer:

\begin{definition} Let $\Ccal$ be a structured category. Given $S$ a class of arrows in $\Ccal$ the \emph{localizer generated by $S$} is smallest localizer $\Wcal(S)$ on $\Ccal$ containing $S$.  A localizer is said to be \emph{small-generated} if it is of the form $\Wcal(S)$ for $S$ a set of arrows. 
\end{definition}

The localizer $\Wcal(S)$ generated by $S$ exists because an arbitrary intersection of localizer is a localizer. This argument works if $\Ccal$ is a small category (for example if large categories are handled using a Grothendieck universes), but in pure ZFC, this argument does not directly apply to a large category. It seems that a more complicated ``predicative'' construction of the localizer generated by a class can be made in ZFC, but we will not detail this. In any cases, our proof of theorem \ref{main_th:Quillen} and \ref{main_th:Quillen_refined} will show in particular that, under the assumptions of these theorems, the localizer generated by any set $S$ exists.

Our main result is:

\begin{theorem}\label{main_th:Quillen} Let $\Ccal$ be a tractable combinatorial structured category. Then a class of arrows $\Wcal$ in $\Ccal$ is the class of equivalences of a combinatorial Quillen model structure on $\Ccal$ if and only if $\Wcal$ is a small-generated $\Ccal$-localizer.
\end{theorem}

Note that, here and everywhere else in the paper, if $\Ccal$ is a structured category, when we talk about a model structure on $\Ccal$ we implicitly means ``with the prescribed class of cofibrations and anodyne fibrations''.\footnote{If we wanted to refer to an arbitrary model structure on $\Ccal$ we would say ``a model structure on the underlying category of $\Ccal$''.} By Smith's theorem, our result can be rephrased as:

\begin{theorem} A localizer $\Wcal$ on a tractable combinatorial structured category $\Ccal$ satisfies the solution set condition if and only if it is small-generated.\end{theorem}

Tractability is a very mild assumption (the author is not aware of any usual model structure that is not tractable). We moreover believe that this assumption can be removed altogether, but we have not been able to do so at this point. In the hope that this will be solved someday, we mention the best result in this direction we have been able to obtain with our method:

\begin{theorem}\label{main_th:Quillen_refined} Given a combinatorial structured category $\Ccal$ and $\Wcal=\Wcal(S)$ a small-generated localizer on $\Ccal$, then $\Wcal$ is the class of equivalence of a Quillen model structure on $\Ccal$ if and only if there is a set $J_0$ of arrows such that:

  \begin{itemize}
  \item All arrows in $J_0$ are cofibrations in $\Wcal$.
  \item If a map $p$ in $\Ccal$ has the right lifting property against all cofibrations between cofibrant objects and all maps in $J_0$, then $p$ is an anodyne fibration.
  \end{itemize}

\end{theorem}

The tractability assumption exactly means that a map with the lifting property against cofibrations between cofibrant objects is an anodyne fibration, so \cref{main_th:Quillen} is the special case of \cref{main_th:Quillen_refined} with $J_0=\emptyset$.

As the existence of such a $J_0$ is a necessary and sufficient condition, it might seems that this completely settle the question of whether each small-generated localizer is the class of equivalences of a Quillen model structure. However we do not know any example where there is no such set $J_0$, and it is still possible that this condition could be removed. We expect it to be difficult to come up with examples where this assumption fails, first because essentially all interesting examples are tractable, hence $J_0 = \emptyset$ works, and also because, as we will discuss below, it can be shown that such a $J_0$ always exists assuming the very strong large cardinal axiom known as Vop\v{e}nka's principle.

\begin{example} A key example is the minimal $\Ccal$-localizer $\Wcal(\emptyset)$. Our theorem imply that, when $\Ccal$ is tractable and combinatorial, it is the class of equivalences of a model structure on $\Ccal$, called the minimal model structure on $\Ccal$. The model structure of this form are exactly the ``Left induced model structure'' as in \cite[Definition 2.1]{rosicky2003left}. This special case of our theorem hence answer the open question of showing that any cofibrantly generated class of cofibrations on a locally presentable category admits a left induced model structure, at least in the case of a tractable class of cofibrations. 
\end{example}

\begin{remark} The existence of left induced model structure for combinatorial structured categories was already showed in \cite{rosicky2003left} as Theorem 2.2, without the tractability assumption, but assuming a very strong large cardinal axiom known as Vop\v{e}nka's principle. Indeed, under Vop\v{e}nka's principle, every localizer on a locally presentable category satisfy the solution set conditions and hence the result follows from Smith's theorem. As this apply to all localizer it also shows that assuming Vop\v{e}nka's principle, \ref{main_th:Quillen} can be proved with no tractability assumptions (and without the small-generated assumption), and with a much simpler proof that what we will give here. 

However, even if the reader is willing to assume Vop\v{e}nka's principle, our result can be seen as a considerable improvement on the bound we obtain on the presentability rank of the corresponding model structure. Our proof gives a fairly concrete bound on this presentability rank described in \cref{ass:on_kappa}. On the other hand the argument above (assuming Vop\v{e}nka principle) only shows that it is it $\kappa$-combinatorial for $\kappa$ the first ``Vop\v{e}nka cardinal'' larger than all the involved data, which is an immensely large cardinal.

\end{remark}

\begin{remark}\label{rk:no_Bousfield_loc} \Cref{main_th:Quillen,main_th:Quillen_refined} can give the false impression to the reader that all left Bousfield localization of combinatorial Quillen model category exist: Indeed if $\Ccal$ is a combinatorial model structure then its class of equivalences is of the form $\Wcal(S)$ for some set $S$, if we want to further invert a set $T$ of arrows we can consider the localizer $\Wcal(S \cup T)$. It will always satisfies the assumption of \ref{main_th:Quillen_refined} because it contains $\Wcal(S)$ which does. So this produces a new model structure $\Ccal'$ on the same underlying structured category as $\Ccal$, i.e.  with the same cofibrations as $\Ccal$, and which is minimal for making all arrows in $T$ into weak equivalences. So this does look a lot like a Bousfield localization of $\Ccal$ at $T$.  In fact, as they have the same cofibrations, $\Ccal'$ is indeed a left Bousfield localization of $\Ccal$. But the point here is that $\Ccal$ does not have to be the left Bousfield localization of $\Ccal$ \emph{at the set $T$}. More precisely:

  \begin{itemize}
  \item The fibrant objects of $\Ccal'$ are not necessarily the $T$-local fibrant objects of $\Ccal$.
  \item The $\infty$-category associated to $\Ccal'$ might not the (co-continuous) localization of the one attached to $\Ccal$ at $T$.
  \end{itemize}

Of course, if the left Bousfield localization of $\Ccal$ at $T$ exists, then it coincide with $\Ccal'$, and $\Ccal'$ has the properties above. But, in general $\Ccal'$ can be a Bousfield localization at a larger set of arrows than $T$. We invite the reader who want to get a better understanding of this to look at what happen on Reid Barton's simple of example of a non-existing left Bousfield localization that we reproduced as example 5.4 in \cite{henry2019CombWMS}. Informally, in the construction of the localizer generated by $T$, one can sometime derive that some maps are in the localizer using pushout along cofibrations between non-cofibrant objects that are not homotopy pushouts. Because these pushouts are not homotopy pushouts, it means that potentially these maps have no good ``homotopy theoretic'' reasons to be weak equivalences in the localization, and hence the localizer might contains map that shouldn't get inverted in a left Bousfield localization.

In \cref{sec:left_semi_localizer}, we will present a version of \cref{main_th:Quillen} for left semi-model categories. As left Bousfield localization of combinatorial left semi-model categories always exists  (see \cite{batanin2020Bousfield} or \cite{henry2019CombWMS}), in this case the theory boils down to constructing the minimal (or left induced) left semi-model category on $\Ccal$, and then taking a left Bousfield localization of it. In this case the corresponding notion of localizer is only stable under pushout of cofibrations when all the objects involved are cofibrant. As such pushouts are always homotopy pushouts, the problem explained above disappear.
\end{remark}

\section{Overview of the proof}

In \cite{cisinski2002theories} and \cite{cisinski2006prefaisceaux}, D.-C.~Cisinski famously proved the very important special case of \cref{main_th:Quillen} where $\Ccal$ is a Grothendieck topos and cofibrations are the monomorphisms (in particular all objects are cofibrant, so these are tractable structured categories). Cisinski's proof of the result can be separated into two parts:

\begin{enumerate}

\item First he shows that (in this special case of a topos and monomorphisms) one can construct explicitly a model structure out of a class of cofibrations and a well behaved cylinder functor.

\item He observes that taking the product by the subobject classifier of $\Ccal$ always gives such a well behaved functor, which works for any localizer.

\end{enumerate}

The first part generalizes relatively well to arbitrary structured categories. This has been done by M.~Oslchok in \cite{olschok2011left} for combinatorial structured categories where every object is cofibrant, and by the author in \cite{henry2019CombWMS} for arbitrary combinatorial structured categories, but at the cost of only constructing a left semi-model category. In this case a ``well behaved cylinder functor'' means a ``strong pre-Quillen cylinder'' in the terminology of \cite{henry2019CombWMS} or a ``cartesian cylinder'' in the terminology of \cite{olschok2011left}.

However, constructing such cylinder functor for large classes of structured categories is delicate, and we do not think it is possible to do it for all combinatorial structured category, or even under some other mild assumptions. So this approach does not seem suitable to prove \cref{main_th:Quillen}, we need a new method.

\bigskip

Our strategy to prove \cref{main_th:Quillen} can instead be summarized as follows: Given a small-generated localizer $\Wcal=\Wcal(S)$, we consider $\kappa$ a ``large enough\footnote{See \cref{ass:on_kappa} for the precise size.} cardinal'' and we take $J$ to be the set of all cofibrations between $\kappa$-presentable objects that are in $\Wcal$. We then use $J$ as our set of generating anodyne cofibrations and show that this generates a Quillen model structure, and we check that its class of weak equivalences coincide with $\Wcal(S)$.

The hard part in the above argument is of course to show that this set $J$ indeed generates a Quillen model structure on $\Ccal$: By construction of $J$ we only have a good control on what happen in the full subcategory $\Pr_\kappa \Ccal \subset \Ccal$ of $\kappa$-presentable objects. So we somehow need to show that $\Ccal$ has a Quillen model structure relying only on properties of $\Pr_\kappa \Ccal$. The first result we know of that allows to do this is due to Zhen Lin low in \cite{low2016heart} and can be phrased as:

\begin{theorem}\label{th:Zhen_Lin} If $\Ccal$ is locally $\kappa$-presentable and the category $\Pr_\kappa \Ccal$ of $\kappa$-presentable objects of $\Ccal$ carries a model structure, then there is a model structure on $\Ccal$ whose fibrations (resp.  anodyne fibrations) are the maps with the lifting properties against anodyne cofibrations (resp. cofibrations) in $\Pr_\kappa \Ccal$. Any combinatorial model structure on $\Ccal$ is obtained in this way from a model structure on $\Pr_\kappa \Ccal$ for some regular cardinal $\kappa$. \end{theorem}

However, while we can (and will) chose a $\kappa$ such that the (cofibration,anodyne fibration) preserves $\kappa$-presentable objects, it does not seem possible to chose a $\kappa$ such that, with $J = \left( \Wcal \cap \Pr_\kappa \Ccal^\to \cap \cof \right) $ as defined above, the (anodyne cofibrations,fibration) weak factorization system restricts to a weak factorization system on $\kappa$-presentable objects. So Low's theorem is not sufficient for our purpose.

Instead we will rely on \cref{th:wms_from_small,th:left_semi_ms_from_small} below. They are finer results, in the same spirit as Low's theorem, but whose assumptions only involves cofibrations and anodyne cofibrations between $\kappa$-presentable objects and do not require explicitly the existence of weak factorization system on $\Pr_\kappa \Ccal$. Unfortunately, at this point we do not know how to formulate similar results directly for Quillen model structures, but only for ``weak model structures'' (A weakening of the notion introduced by the author in \cite{henry2018weakmodel} and \cite{henry2019CombWMS}) and for left semi-model structure (as introduced by Spitzweck in \cite{spitzweck2001operads}). A brief summary of the theory of weak model categories will be given in \cref{sec:wms}.

\begin{theorem}\label{th:wms_from_small} Let $\Ccal$ be a $\kappa$-combinatorial structured category, and let $J$ be a class of cofibrations between $\kappa$-presentable objects which satisfies the following conditions

\begin{enumerate}[label=(\roman*)]

\item $J$ contains isomorphisms, is stable under pushout and $\kappa$-small transfinite compositions. 

\item Given two composable cofibrations $A \overset{j}{\cto} B \overset{i}{\cto} C$ between $\kappa$-presentable cofibrant objects. If $j$ and $ i \circ j$ are in $J$, then so does $i$.

\item For any cofibration $i : A \cto B$ between $\kappa$-presentable cofibrant objects, there is a diagram of $\kappa$-presentable objects of the form below (using \cref{notation:cofibration_cof_square})
\begin{equation} 
\begin{tikzcd}
  A \ar[r,cto,"i"] \ar[dr,phantom,"\ulcorner"{name=M,description}]\ar[d,cto] & B \ar[d,cto,"\in J"] \\
B' \ar[r,cto,"\in J"{swap}] & C \ar[from = M,cto]
\end{tikzcd}
\end{equation}

\end{enumerate}

Then $\Ccal$ admits a weak model structure with its given class of cofibrations, and $J$ as sets of generating anodyne cofibrations.

\end{theorem}

\begin{notation}\label{notation:cofibration_cof_square} In this paper, we use $\cto$ to denote cofibrations in diagrams. A \emph{cofibrant square} is a square, denoted as

\[ \begin{tikzcd}
  A \ar[r,cto] \ar[dr,phantom,"\ulcorner"{name=M,description}]\ar[d,cto] & B \ar[d,cto] \\
B' \ar[r,cto] & C \ar[from = M,cto]
\end{tikzcd} \]

where all maps are cofibrations and the comparison map $B \coprod_A B' \to C$ is a cofibration as well. It is easy to see that cofibrant square can be composed, that is given two cofibrant squares
\[ \begin{tikzcd}
  A \ar[r,cto] \ar[dr,phantom,"\ulcorner"{name=M,description}]\ar[d,cto] & B \ar[d,cto] \ar[r,cto] \ar[dr,phantom,"\ulcorner"{name=M2,description}] & C \ar[d,cto] \\
D \ar[r,cto] & E \ar[from = M,cto] \ar[r,cto] & F \ar[from = M2,cto]
\end{tikzcd} \]
\noindent the outer rectangle is also a cofibrant square.

\end{notation}

Building on \cref{th:wms_from_small}, we will also prove a version for left semi-model categories:

\begin{theorem}\label{th:left_semi_ms_from_small} Let $\Ccal$ be a $\kappa$-combinatorial structured category, and let $J$ be a class of cofibrations between $\kappa$-presentable objects which satisfies the following conditions

\begin{enumerate}[label=(\roman*)]

\item $J$ contains isomorphisms, is stable under pushout and $\kappa$-small transfinite compositions. 

\item Given two composable cofibrations $A \overset{j}{\cto} B \overset{i}{\cto} C$ between $\kappa$-presentable cofibrant objects. If $j$ and $ i \circ j$ are in $J$, then so does $i$.

\item[(iii')] For any cofibration $i : A \cto B$ between $\kappa$-presentable cofibrant objects, there is a cofibrant square of $\kappa$-presentable objects of the form

\begin{equation}
\begin{tikzcd}
  A \ar[r,cto,"i"] \ar[dr,phantom,"\ulcorner"{name=M,description}]\ar[d,cto] & B \ar[d,cto,"\in J"description] \\
B' \ar[r,cto,"\in J"{swap}] & C \ar[from = M,cto] \ar[u,bend right=30,dotted,"r"swap]\\
\end{tikzcd}
\end{equation}

where $B \cto C$ admit a (dotted) retraction $r$.


\end{enumerate}

Then $\Ccal$ admits a Fresse left semi-model category with its given class of cofibrations, and where fibrant object and fibrations between fibrant objects are characterized by the right lifting property against $J$.

\end{theorem}

Versions of \cref{th:wms_from_small,th:left_semi_ms_from_small} have been first observed by the author while working on \cite{henry2020model}. In \cite{henry2020model} we construct various model structures on the category of premodel categories and $\kappa$-combinatorial premodel categories\footnote{As in item \ref{item:def_premodel} and \ref{item:def_comb_premodel} of \cref{sec:wms}.)} generalizing the one constructed by R.~Barton in the simplicial case in \cite{barton2020model}. In these model structures fibrant objects are various flavor of model categories (weak, left etc...) and equivalences between fibrant object are the Quillen equivalences. In this context, \cref{th:wms_from_small,th:left_semi_ms_from_small} appears as the characterization of fibrant objects amongst $\kappa$-combinatorial categories. Indeed, as morphisms of $\kappa$-combinatorial categories are taken to be left adjoint functor preserving $\kappa$-presentable objects, lifting property in this category only express properties of the $\kappa$-presentable objects, and so characterizing model structure as the fibrant objects is a statement very similar to these two theorems. However, the proofs we will give in \cite{henry2020model} are deeply interconnected with the construction of these model structure and will follows a much more conceptual approach. Given the importance of these results toward the proof of our main theorem, we chose to give a direct and more elementary proof of these results here.

Our argument for \cref{th:wms_from_small,th:left_semi_ms_from_small} rely in an essential way on the ``fat small object argument'' from \cite{makkai2014fat} (originally due to J.~Lurie in \cite{lurie2009higher}). Informally, the fat small objects argument will allow us to propagate the properties of the class $J$ assumed in \cref{th:left_semi_ms_from_small,th:left_semi_ms_from_small} to similar property of the class of $J$-cellular arrows, that is the maps that are transfinite composition of pushouts of arrows in $J$ ($J$-cofibrations being the retracts of theses).

In \cref{section_fat_small_object_argument} we review the fat small object argument, or rather a reformulation of it in the language of Reedy diagrams. Essentially this argument  allows to represent arbitrary $J$-cellular maps as colimits of $J$-Reedy maps between diagrams indexed by well-founded posets satisfying smallness conditions that allows for inductive reasoning. In \cref{sec:extention_to_Reedy} we show that the assumptions of \cref{th:wms_from_small} or \cref{th:left_semi_ms_from_small} can be propagated to these Reedy maps between diagrams. In \cref{sec:MS_from_small} we combine the two to show that these assumptions can be propagated from $J$ to $J$-cellular maps. Using results from \cite{henry2018weakmodel} and \cite{henry2019CombWMS} this will be enough to construct a weak model structure or a left semi-model category on $\Ccal$ and hence proving \cref{th:wms_from_small,th:left_semi_ms_from_small}. Because the fat small object argument works more naturally with cellular maps rather than with cofibrations, we will actually prove slightly more general statement that uses ``cellular categories'' (i.e. categories endowed with a class of cellular maps) instead of structured categories.

The argument for \cref{th:wms_from_small,th:left_semi_ms_from_small} from \cite{henry2020model} will also rely on the fat small object argument, but in a considerably less explicit way: it is hidden in the proof that the inclusion functor from $\lambda$-combinatorial premodel categories to $\kappa$-combinatorial premodel categories for $\lambda \leqslant \kappa$ preserves anodyne $\Wb$-fibrations. This is then used in an essential way in the proof of several key results. The author believe, though with not much conviction, that a completely explicit unfolding of the more conceptual proof we will give in \cite{henry2020model} would be in the end quite similar to the more direct and technical proof we give here. 

\bigskip

With these two theorems, we will be able to prove \cref{main_th:Quillen,main_th:Quillen_refined} in \cref{sec:localizer} using ideas somehow inspired from Low's result (i.e. \cref{th:Zhen_Lin}). As the (cofibration,anodyne fibrations) factorization system is fixed it is possible to find a cardinal $\kappa$ such that the (cofibration,anodyne fibration) factorization preserves $\kappa$-presentable objects (see \cref{rk:SOA_access_rank}). Taking $J$ to be the set of cofibrations in $\Wcal(S)$ between $\kappa$-presentable objects we will be able to show that $J$ satisfies the conditions of \cref{th:left_semi_ms_from_small} and hence generates a left semi-model structure on $\Ccal$. We then conclude that this is a Quillen model structure by directly showing that anodyne cofibrations are weak equivalences: By the fat small objects argument, arbitrary $J$-cellular maps are $\kappa$-filtered colimits of maps in $J$, which are equivalences essentially by constructions, and we will show that weak equivalences are closed under $\kappa$-filtered colimits. This essentially finish the proof of \cref{main_th:Quillen,main_th:Quillen_refined}.

Finally, in \cref{sec:left_semi_localizer}, we will sketch a version of \cref{main_th:Quillen} for left semi-model categories.

\section{Weak model structures and notations}
\label{sec:wms}

In this section, we briefly recall some terminology regarding weak model structures, mostly following \cite{henry2018weakmodel} and \cite{henry2019CombWMS}.

\begin{enumerate}

\item\label{item:def_premodel} A \emph{pre-model category} is a category with two weak factorization systems (cofibrations, anodyne fibrations) and (anodyne cofibrations, fibrations) such that anodyne cofibrations are cofibrations. We use $\cto$ and $\fto$ to denote cofibrations and fibrations and the symbol $\ano$ to denote anodyne (co)fibrations.

\item\label{item:def_comb_premodel} A pre-model category is said to be \emph{combinatorial} if its underlying category is locally presentable and its factorization systems are cofibrantly generated.

\item\label{item:def_acyclic_cof} A \emph{core fibration} is a fibration with fibrant target, (and hence fibrant domain as well). Dually, a \emph{core cofibration} is a cofibration with cofibrant domain.

\item\label{item:def_acyclic_fib} An \emph{acyclic fibration} is a fibration that has the right lifting property against core cofibrations. An \emph{acyclic cofibration} is a cofibration with the left lifting property against core fibrations. Acyclic (co)fibrations are indicated in diagram by the symbol $\sim$.

\item\label{item:def_saturation} A pre-model category is \emph{left saturated} if all acyclic cofibrations are anodyne cofibrations and \emph{right saturated} if all acyclic fibrations are anodyne fibrations. One says that a premodel category is \emph{core left saturated} if only the core acyclic cofibrations are anodyne cofibrations, and \emph{core right saturated} if only the core acyclic fibrations are anodyne fibrations. Section 4 of \cite{henry2019CombWMS} provides various universal constructions, called saturation, that turn a combinatorial pre-model category in a left/right saturated one, without affecting its ``core'', that is the fibration between fibrant objects and cofibrations between cofibrant objects, nor any homotopy theoretic property. A Quillen model category is automatically both left and right saturated.

\item \label{item:def_wms} A pre-model category is said to be a \emph{weak model category} if each cofibration $A \cto B$ between bi-fibrant objects admit a ``relative strong cylinder object'', that is a factorization of the codiagonal map
\[ B \coprod_A B \cto I_A B \to B \]
where the first map (equivalently both map) $B \cto I_A B$ is an acyclic cofibration, and if we have the dual condition for fibrations between bi-fibrant objects.

\item It was shown in \cite{henry2018weakmodel} that if $\Ccal$ is a weak model category, then we can define a homotopy relation between maps, construct a homotopy category as a category of bifibrant objects with homotopy classes of maps between them, obtain a well behaved notion of weak equivalences (at least for arrows between fibrant or cofibrant objects), show that core (co)fibration are acyclic if and only if they are weak equivalences, and more generally to develop a well behaved homotopy theory very similar to what happen with a Quillen model category.

\item\label{item:weak_cylinder} The existence of relative strong cylinder object for cofibrations between bifibrant objects mentioned above can be equivalently replaced by the existence of a \emph{relative weak cylinder object} for each core cofibration $A \cto B$, that is a diagram
\[\begin{tikzcd}
 B \coprod_A B \ar[r,cto] \ar[d,"\nabla"] & I_A B \ar[d] \\
B \ar[r,cto,"\sim"] & D_A B
\end{tikzcd}
\]
where $\nabla$ denotes the codiagonal map, and in addition the first map $B \cto I_A B$ induced by the top map is an acyclic cofibration. Strong cylinder objects corresponds exactly to the case where the map $B \cto D_A B$ is an isomorphism (or just admit a retraction).

\item\label{item:self_composed_span} Lemma~2.3.6 and~Remark 2.3.7 of \cite{henry2018weakmodel} give a trick that allows to turn a diagram as in \cref{th:wms_from_small}.$(iii)$ on the left below

\[\begin{tikzcd}
  A \ar[r,cto,"i"] \ar[dr,phantom,"\ulcorner"{name=M,description}]\ar[d,cto] & B \ar[d,cto,"\sim"]  & & \displaystyle B \coprod_A B \ar[r,cto] \ar[d] & \displaystyle C \coprod_{B'} C \ar[d] \\
B' \ar[r,cto,"\sim"{swap}] & C \ar[from = M,cto] & & B \ar[r,cto,"\sim"] & C
\end{tikzcd}
\]
into a relative weak cylinder object for $i : A \cto B$, with $I_A B= C \coprod_{B'} C$ and $D_A B = C$, as drawn on the right above.
\end{enumerate}

Semi-model category where introduced by Spitzweck in \cite{spitzweck2001operads}, following some previous work by Hovey \cite{hovey1998monoidal}. A slightly more general version of the notion has been introduced by Fresse in \cite{fresse2009modules}. The reader can also look at \cite{henry2019CombWMS} for a presentation of these various definition in relation with weak model categories. In this paper we will only use left and right ``semi-model category'' in a slightly more restricted sense than in most references, this is due to the fact that we are only interested in \emph{combinatorial} left semi-model category. We will use the following definition:

\begin{definition} A \emph{Fresse left semi-model category} is a pre-model category which admits a class $\Wcal$ of maps called weak equivalences such that

\begin{itemize}

\item $\Wcal$ satisfies $2$-out-of-$3$.

\item A fibration is an acyclic fibration if and only if it is in $\Wcal$.

\item A core cofibration is an anodyne cofibration if and only if it is in $\Wcal$.

\end{itemize}

One say it is a \emph{Spitzweck left semi-model category} if fibrations that are in $\Wcal$ are in fact anodyne fibrations.

\end{definition}

So, in such a semi-model category, we have two fully formed weak factorization systems, but anodyne cofibrations coincide with cofibrations that are equivalences only for arrows with cofibrant domain, while in a Quillen model category this would be unconditional. The more traditional definition still uses ``trivial cofibration'' to means the cofibrations that are weak equivalences, but instead weaken the requirement that (trivial cofibration, fibrations) form a weak factorization systems.

\begin{remark}\label{rk:charac_of_left_ms} Theorem 3.7 of \cite{henry2019CombWMS} shows that a pre-model category $\Ccal$ is a Fresse left semi-model category if and only if:

\begin{itemize}

\item $\Ccal$ is a weak model category.
\item $\Ccal$ is core left saturated.
\item  every cofibrant object $X$ in $\Ccal$ admits a strong cylinder object
\[X \coprod X \cto IX \to X \]
 where the map $X \cto IX$ is an acyclic cofibration.
\end{itemize}

Note that in a weak model category, this last condition automatically holds for $X$ bifibrant, but is a non-trivial assumption for a general cofibrant objects, that in general will only admits ``weak cylinder object'' in the spirit of item \ref{item:weak_cylinder} above.

And of course, it is a Spitzweck left semi-model category if it is further more right saturated.

\end{remark}

\begin{remark}\label{rk:More_general_wms} In \cite{henry2018weakmodel}, we considered a slightly more general notion of weak model category where we do not necessarily have weak factorization systems (i.e. is not a pre-model category in the first place). More precisely, we only ask to have a class of fibrations and a class of cofibrations respectively closed under composition and pullback/pushouts, and that an arrow from a cofibrant object to a fibrant object can always be factored as an ``acyclic cofibration'' followed by a fibration and as a cofibration followed by an ``acyclic fibrations'', where acyclic (co)fibration are as defined in point \ref{item:def_acyclic_cof} and \ref{item:def_acyclic_fib} above, and we then require the existence of path objects and cylinder objects as in point \ref{item:def_wms} above.

 This includes the original definition of left and right semi-model categories, which do not have fully formed weak factorizations system. In the present paper, we will use this more general notion only at one point (\cref{th:wms_from_small_cellular}), and only because we will take as ``cofibrations'' a class of cellular maps not closed under retract, as discussed below. \end{remark}

Finally, in a large part of the paper we will need to work with a class of \emph{cellular maps}, not closed under retract, instead of a class of cofibrations. The reason for this is that the fat small object argument is really about cellular maps (instead of cofibrations). To be precise:

\begin{definition}
Given $I$ a set of maps, a $I$-cellular map is a map that is a transfinite composition of pushouts of maps in $I$. A \emph{$\kappa$-combinatorial cellular category} or \emph{$\kappa$-cellular category} is a locally $\kappa$-presentable category, endowed with a class of maps, called \emph{cellular maps}, which is the class of $I$-cellular maps for $I$ a set of arrow between $\kappa$-presentable objects.
\end{definition}

A $\kappa$-cellular category is also automatically a $\kappa$-combinatorial structured category, where ``cofibrations'' are exactly the retracts of cellular maps and anodyne fibration the map with the right lifting property against all cellular maps (hence all cofibrations as well). A \emph{cellular object} is an object $X$ such that the unique map $\emptyset \to X$ is a cellular map.

\begin{remark}\label{rk:cellular_are_enough} The proofs we will give of \cref{th:wms_from_small,th:left_semi_ms_from_small} in \cref{sec:MS_from_small} apply slightly more generally when structured categories are replaced with cellular categories, and cofibration with cellular maps everywhere. \Cref{th:wms_from_small,th:left_semi_ms_from_small} corresponds to special cases when assumption of closure under retract are added.\end{remark}

\begin{notation} We define \emph{cellular squares} exactly as the cofibrant squares of \cref{notation:cofibration_cof_square}, except that all the maps involved are now required to be cellular maps instead of cofibrations. We still use the same notation $\cto$ for cellular maps and cellular square will be denoted as in \cref{notation:cofibration_cof_square}. \end{notation}

We also recall the following:

\begin{prop}\label{prop:cof_are_cellular} Let $\kappa$ be an uncountable regular cardinal, and let $\Ccal$ be a $\kappa$-combinatorial structured category. Let $I$ be the set of all cofibrations between $\kappa$-presentable objects. Then the $I$-cellular morphisms coincide with the cofibrations. 
\end{prop}

To put it another way, if the generating set $I$ is closed under $\kappa$-small transfinite composition, pushouts and \emph{retract}, then the class of $I$-cellular morphisms is also closed under retract. This was first proved by J.~Lurie as Proposition A.1.5.12 of \cite{lurie2009higher}. In fact this result is the main reason why Lurie introduced his notion of ``good colimits'', and what we call the fat small object argument. This results is also reproved in \cite{makkai2014fat} as Theorem~B.1.

\section{The fat small object argument}
\label{section_fat_small_object_argument}
In this section, we review the fat small object argument and the technology of ``good colimits'' introduced in Appendix A.2.6 of \cite{lurie2009higher} and presented in more details in \cite{makkai2014fat}. We present these results with a slightly different perspective that will be more convenient for the applications we have in mind, but none of the result of this section are new, they can all be found in \cite{lurie2009higher} and \cite{makkai2014fat}, eventually after some translation.

We fix $\kappa$ a regular cardinal and $\Ccal$ a $\kappa$-cellular category, with an explicitly chosen set of generators $I$. That is, $\Ccal$ is locally $\kappa$-presentable and $I$ is a set of arrows between $\kappa$-presentable objects. We will say that a map is $\kappa$-cellular (or $I$-$\kappa$-cellular) if it is a $\kappa$-small transfinite composition of pushouts of maps in $I$. Note that as $\kappa$-presentable objects are closed under $\kappa$-small colimits, the target of any $\kappa$-cellular map with $\kappa$-presentable domain is also $\kappa$-presentable. In particular, $\kappa$-cellular objects are $\kappa$-presentable.

\begin{notation} In a poset $P$ we write $x < y$ for $x \leqslant y$ and $x \neq y$. We define
\[ \partial p  \coloneqq \{ y  | y < p \}. \]
 $P$ is said to be \emph{well-founded} if it satisfies the induction principle, i.e. if given $U \subset P$, such that for all $p \in P$, $ \partial p \subset U \Rightarrow p \in U$, we have $U = P$.

For $X \subset P$ we denote:

\[ \sievec X =\{ p \in P | \exists x \in X, p \leqslant x\} \]

A subset $X \subset P$ is said to be a \emph{sieve} in $P$ if $\sievec X = X$, i.e. if $y \leqslant x$ and $x \in X$ imply $y \in X$. We write $X \sievei P$ when $X$ is a sieve in $P$. For example, for each $p \in P$ we have sieve inclusions $\partial p \sievei \sievec p \sievei P$.
A poset $P$ is considered as a category, with a unique map from $x$ to $y$ if $x \leqslant y$. We denote by $\Ccal^P$ the category of $P$-diagram in $\Ccal$, i.e. functors $P \to \Ccal$.
\end{notation}

\begin{definition}\label{def:locally_lambda_small} A well-founded poset $P$ is said to be \emph{locally $\kappa$-small} if for all $p \in P$ the sieve $\sievec p$ (or equivalently $\partial p$), has cardinality strictly smaller than $\kappa$.
 \end{definition}

In the terminology of \cite{lurie2009higher} and \cite{makkai2014fat} a ``$\kappa$-good'' poset is a locally $\kappa$-small well-founded poset with a minimal element. With our slightly modified presentation of the argument, the minimal element is no longer really important and we will work with locally $\kappa$-small well-founded posets. We also want to avoid the very generic terminology ``good''.

\begin{notation}\label{nota:latching_map} For any sieve $U \subset P$, and $\Xcal \in \Ccal^P$ a diagram, one denotes:

\[ \Xcal(U) \coloneqq \colim_{u \in U} \Xcal(u) \]

Where $U$ is seen as a full subcategory of $P$. In particular, if $p \in P$ we have%
\[ X(p) = X(\sievec p ) \qquad \Xcal(\partial p ) = \colim_{p' < p } \Xcal(p') \]

The canonical map $\Xcal(\partial p) \rightarrow \Xcal(p)$ induced by the sieve inclusion $\partial p \sievei \sievec p$, is called the \emph{latching map} of $\Xcal$ at $p$. Given an arrow $f: \Xcal \to \Ycal$ between two $P$-diagrams, the map \[ \Xcal(p) \coprod_{\Xcal(\partial p)} \Ycal(\partial p) \rightarrow \Ycal(p) \] is called the \emph{latching map} of $f$ at $p$. We have the following easy fact: \end{notation}


\begin{lemma}\label{lem:transfini_compo_of_latching} If $f:\Xcal \to \Ycal$ is a morphism of $P$-diagram, for $P$ a well founded poset. And $U \sievei V \sievei P$ are two sieves, then the natural map
\[ \Xcal(V) \coprod_{\Xcal(U)} \Ycal(U) \to \Ycal(V) \]
is a transfinite composite of pushouts of the latching map of $f$ at each $x \in V- U$. Each latching map appears exactly once in the composite.
\end{lemma}

Two special cases of interest: if one takes $U = \emptyset$, we get that the map $\Xcal(V) \to \Ycal(V)$ is a transfinite composition of the pushouts of all the latching map of $f$ at all $v \in V$, and if we take $\Xcal$ to be constant at the initial object, then we get $\Ycal(U) \to \Ycal(V)$ is a transfinite composition of pushouts of all the latching maps of $\Ycal$ at $v \in V-U$. In both case each latching map appearing exactly once.

\begin{proof}By a classical set theoretic argument we can chose an order preserving bijection $\pi:\alpha \to V$ where $\alpha$ is an ordinal and $\beta=\pi^{-1} U$ is an initial segment of $\alpha$ (hence also an ordinal). For each $i \in \alpha, i \geqslant \beta$ we consider
\[ Y_i \coloneqq \Xcal(V) \coprod_{\Ycal(\pi(\sievec i))} \Ycal(\pi(\sievec i)) \]
which defines a functor on $\alpha - \beta$ with values in $\Ccal$. We then easily check by direct computation that:
\begin{itemize}
\item $Y_\beta = \Xcal(V) \coprod_{\Xcal(U)} \Ycal(U)$.
\item $Y_\alpha = \Ycal(V)$.
\item At a limit ordinal $\lambda$, with $\beta < \lambda < \alpha$ the map $ \colim_{\gamma < \lambda} Y(\gamma) \to Y(\lambda)$ is a pushout of the latching map of $f$ at $\pi(\lambda)$
\item A a successor ordinal $\lambda = \gamma^+$ with $\beta <\lambda < \alpha$  the map $Y_\gamma \to Y_\lambda$ is a pushout of the latching map of $f$ at $\pi(\lambda)$.
\end{itemize}
These four facts together prove the result.
\end{proof}

\begin{definition} A morphism $\Xcal \rightarrow \Ycal$ of diagrams will be called a \emph{$I$-Reedy} map if for all $p \in P$ the latching map is either an isomorphism or a pushout of a map in $I$. 
\end{definition}

In the particular case $\Xcal = \emptyset$, one says that $\Ycal$ is a $I$-Reedy diagram if $\Ycal(\partial p) \to \Ycal(p)$ is an isomorphism or a pushout of a morphism in $I$. By \cref{lem:transfini_compo_of_latching}, if $\Xcal \rightarrow \Ycal$ is a $I$-Reedy morphism then the natural map:
\[ \Xcal(P) = \colim_{p \in P} \Xcal(p) \rightarrow \Ycal(P) = \colim_{p\in P} \Ycal(p) \]
is a $I$-cellular morphism. If in addition $P$ has cardinality $<\kappa$, then the map above is $\kappa$-cellular. In particular, for any locally $\kappa$-small well-founded poset $P$, if $\Xcal$ is a $I$-Reedy diagram $P \to \Ccal$ then all the $\Xcal(p)$ and $\Xcal(\partial p)$ appearing above are $\kappa$-cellular.

\begin{notation} \label{notation:extention_of_well_founded_poset} Given well-founded posets $Q$ and $P$ with a sieve inclusion $P \sievei Q$, we say that $Q$ is an extension of $P$. If $V \sievei P$ is any sieve, there is a unique extension $P^+$ of $P$ by a single element $x$, such that $\partial x=V$.
Note that $P^+$ is locally $\kappa$-small, as in \cref{def:locally_lambda_small}, if and only if $P$ is locally $\kappa$-small and $V$ is $\kappa$-small.
\end{notation}


\begin{construction} \label{cstr:canonical_extention} If $P \sievei Q$ is a sieve inclusion, and $\Xcal$ is a $P$-diagram, we denote by $\Xcal^Q$ the (pointwise) left Kan extension of $\Xcal$ to $Q$. The formula for pointwise Kan extention boils down in this special case to:
\[ \Xcal^Q(q) = \Xcal( P \cap \sievec q ) \]

It generalizes to any sieve $U \sievei Q$ to
\[\Xcal^Q(U) = \Xcal(U \cap P) \]
And hence satisfies $\Xcal^Q(p) = \Xcal(p)$ for all $p \in P$, and more generally $\Xcal^Q(U) = \Xcal(U)$ for all $U \sievei P$. In particular, the latching map of $\Xcal^Q$ at any $q \in P$ is the latching map of $\Xcal$ at $p$, and its latching map at $q \in  Q- P$ is an isomorphism. This implies that a map $f: \Xcal \to \Ycal$ of $P$-diagram is a $I$-Reedy if and only if its Kan extension $f^Q:\Xcal^Q \to \Ycal^Q$ is $I$-Reedy.

We will very often simply write $\Xcal$ instead of $\Xcal^Q$ to denote the Kan extension and identify $P$-diagrams with their Kan extension as $Q$-diagrams. The observations above and the following easy lemma confirm that this can be done without too much harm.
\end{construction}

\begin{lemma}
For a sieve inclusion $P \sievei Q$, taking left Kan extension identifies the category $\Ccal^P$ of $P$-diagrams with the full subcategory of $Q$-diagrams whose latching maps are isomorphisms at all $q \in Q-P$. 
\end{lemma}

\begin{proof}
Kan extension is left adjoint to restrictions. The unit of adjunction is $\Xcal(p) \to \Xcal^Q(p) = \Xcal(P \cap \sievec p ) = \Xcal(p)$, and is an isomorphism, which shows that taking Kan extension is fully faithful. We have already explained above that objects in the image have isomorphic latching map at all $q \in Q-P$. Conversely, if $\Ycal \in \Ccal^Q$ satisfies this condition on latching maps, then the comparison map $\Ycal(U \cap P) \to \Ycal(U)$ is an isomorphisms because it is a transfinite composition of pushouts of the latching maps at all $u \in U - P \subset Q - P$, which are all isomorphisms. This shows that such a diagram $\Ycal$ is the Kan extension of its restriction to $P$.
\end{proof}

\begin{construction}\label{rk:S_lambda_P} If $P$ is a locally $\kappa$-small well-founded poset then $P$ is the $\kappa$-filtered union of all its $\kappa$-small sieves. More precisely, we consider the poset
 \[\Scal_\kappa P \coloneqq \{ \text{$S \sievei P$, $\kappa$-small sieves}\} \]
of sieve ordered by inclusion. For any $P$-diagram $\Xcal$ we have
\begin{equation}\label{eq:dir_colimit} \Xcal(P) \simeq \colim_{V \in \Scal_\kappa P} \Xcal(V). 
\end{equation}
As $\Scal_\kappa P$ is closed under $\kappa$-small union, the colimit above is $\kappa$-filtered. In particular, if $\Xcal$ takes values in $\Pr_\kappa \Ccal$, for example if $\Xcal$ is a $I$-Reedy diagram, then this colimit is a $\kappa$-filtered colimits of $\kappa$-presentable objects.
\end{construction}

We are now ready to present our version of the fat small object argument:

\begin{lemma} \label{lem:FatSOA} Let $I$ be any set of arrows between $\kappa$-presentable objects in locally $\kappa$-presentable category $\Ccal$. Let $P$ be a locally $\kappa$-small well-founded poset, and $\Xcal: P \rightarrow \Pr_\kappa \Ccal$ by any diagram. Then for any $I$-cellular morphism:

\[ i: \Xcal(P) \overset{i}{\cto} Y \]

There exists a locally $\kappa$-small well-founded extension $P \sievei Q$, and a $I$-Reedy morphism $\Xcal \rightarrow \Ycal$ of $Q$-diagram such that the map above identifies with $\Xcal(P) \simeq \Xcal(Q) \rightarrow \Ycal(Q)$.
\end{lemma}

\begin{proof}
We start with the case where $i$ is a pushout by a single morphism $A \cto B \in I$ along a map $A \to \Xcal(P)$. By \cref{rk:S_lambda_P}, we have
a $\kappa$-filtered colimit \[ \Xcal(P) = \colim_{V \in \Scal_\kappa P} \Xcal(V). \]
As $A$ is $\kappa$-presentable, the map $A \to \Xcal(P)$ factors through $\Xcal(V)$ for some $V \in \Scal_\kappa P$. One then form $Q$ the extension of $P$ obtained by adding a single element $x$ such that $\partial x = V$, and we extend the diagram $\Xcal$ to a $Q$-diagram $\Xcal_1$ by
\[ \Xcal_1(x) = \Xcal(V) \coprod_A B \]
whose latching map at $x$ is the pushout $\Xcal(V) \cto \Xcal(V) \coprod_A B$, and whose latching map at all $p \in P$ is just that of $\Xcal$. We indeed have
\[ \Xcal_1(Q) = \Xcal_1(P) \coprod_A B \]
so that $\Xcal(P) \to \Xcal_1(Q)$ is indeed the map $i$.  If $i$ is a transfinite composite of pushout of morphism in $I$ one just iterate this process adding one element to $Q$ at each stage of the composite, and taking the increasing union of the posets $Q$ at each limit stage. 
\end{proof}

\begin{remark} Given a cellular map $X \cto Y$, one often apply the fat small object argument to it by first finding a representation of $X$ as $\Xcal(P)$ using the fat small object argument on the map $\emptyset \to X$. Note that one can do this even if $X$ is not cofibrant due to the following lemma:
\end{remark}

\begin{lemma}\label{lem:all_maps_are_cellular} If $\Ccal$ is a locally $\kappa$-presentable category and $I$ is the set of all arrows between $\kappa$-presentable objects, then every map in $\Ccal$ is $I$-cellular. In particular, any object of $\Ccal$ is the colimit of a diagram $P \to \Pr_\kappa \Ccal$ for $P$ a locally $\kappa$-small well-founded poset.
\end{lemma}

\begin{proof} If $i :X \to Y$ is a map in $I$, then the codiagonal map $\nabla_i:Y \coprod_X Y \to Y$ is also in $I$. The lifting property against $\nabla_i$ corresponds exactly to the uniqueness in the lifting property against $i$. It hence follow that a map that has the weak lifting property against all maps in $I$, has actually the unique lifting property against all maps in $I$. But as $\kappa$-presentable objects are strong generators of $\Ccal$, a map with the unique lifting property against all maps between $\kappa$-presentable objects is an isomorphism. Hence, the ordinary small object argument factors any map in $\Ccal$ as an $I$-cellular map followed by an isomorphism, which proves the first claim. The second claim follows from \cref{lem:FatSOA} applied to $\emptyset \to X$ for $X$ any object of $\Ccal$ and $\emptyset$ the initial object of $\Ccal$.

\end{proof}

\section{Extension to Reedy diagram}
\label{sec:extention_to_Reedy}
\begin{assumptions}\label{ass:cellular_th_wms_from_small} In this section, we work essentially under the assumption of \cref{th:wms_from_small}, modified according to \cref{rk:cellular_are_enough}. More precisely, we fix $\Ccal$ a $\kappa$-cellular category, and we take $I$ to be the set of all cellular maps between $\kappa$-presentable objects. In particular, cellular maps are the $I$-cellular maps. We consider $J$ a subset of $I$ that satisfies the three conditions of \cref{th:wms_from_small}:

\begin{enumerate}[label=(\roman*)]

\item $J$ contains isomorphisms, is stable under pushout and $\kappa$-small transfinite composition. 

\item Let $A \overset{j}{\cto} B \overset{i}{\cto} C$ be two composable maps in $I$, with $A$ a $\kappa$-cellular object (in particular $i,j \in I$). If $j$ and $ i \circ j$ are in $J$, then so does $i$.

\item For any $i : A \cto B \in I$ with $A$ and $B$ cellular ($\kappa$-presentable) objects, there is a cellular square of $\kappa$-presentable objects of the form:

\[ 
\begin{tikzcd}
  A \ar[r,cto,"i"] \ar[dr,phantom,"\ulcorner"{name=M,description}]\ar[d,cto] & B \ar[d,cto,"\in J"] \\
B' \ar[r,cto,"\in J"{swap}] & C \ar[from = M,cto]\\
\end{tikzcd}
\]

\end{enumerate}

We will at some point, also consider the stronger condition $(iii')$ from \cref{th:left_semi_ms_from_small}, also modified according to \cref{rk:cellular_are_enough}:

\begin{enumerate}
\item[(iii')] For any $i : A \cto B \in I$ with $A$ and $B$ cellular $\kappa$-presentable objects, there is a cellular square as in condition $(iii)$ above where $B \cto C$ admits a retraction $C \to B$.
\end{enumerate}

We also fix $P$ a well-founded locally $\kappa$-small poset as in \cref{def:locally_lambda_small}. Our general goal in this section is to show that the category of functors $P \to \Pr_\kappa \Ccal$ has the same type of properties.  We first have:
\end{assumptions}

\begin{lemma}\label{lem:charac_Reedy_J_morphisms} Under \cref{ass:cellular_th_wms_from_small}, if $f:\Xcal \rightarrow \Ycal$ is a $I$-Reedy morphism of $I$-Reedy $P$-diagrams, then $f$ is a $J$-Reedy morphisms if and only if $\Xcal(p) \rightarrow \Ycal(p)$ is in $J$ for all $p \in P$.
\end{lemma}

\begin{proof} If $f:\Xcal \to \Ycal$ is a $J$-Reedy morphism, then \cref{lem:transfini_compo_of_latching} shows that $\Xcal(p) \to \Ycal(p)$ is a transfinite composition of pushouts of all the latching map at $x \leqslant p$.  Because $J$ is closed under pushout and $\kappa$-small transfinite compositions and $P$ is locally $\kappa$-small this imply that this map is in $J$. Conversely, we assume that $\Xcal(p) \rightarrow \Ycal(p) \in J$ for all $p$, and we prove by induction on $p \in P$ that the latching map of $f$ at $p$ is in $J$. If it has been proved at all $p'<p$, then $\Xcal(\partial p) \rightarrow \Ycal(\partial p)$ is in $J$ (by the argument above restricted to $\partial p \subset P$). Hence as the map $\Xcal(p) \rightarrow \Ycal(p) \in J$ decomposes as
\[
\begin{tikzcd}
  \Xcal(\partial p) \ar[dr,phantom,"\ulcorner"very near end] \ar[r] \ar[d,cto,"\in J"] &  \Xcal(p) \ar[d,cto,"\in J"] \ar[dr,cto,"\in J"] \\
 \Ycal(\partial p)  \ar[r] & \displaystyle \Xcal(p) \coprod_{\Xcal(\partial p)} \Ycal(\partial p) \ar[r,"l \in I"{swap}] & \Ycal(p)
\end{tikzcd}
\]
where $l$ is the latching map at $p$, which is in $I$ because $f$ is a $I$-Reedy morphism, hence, as $\Xcal(p)$ is cellular (because $\Xcal$ is a $I$-Reedy diagram) point $(ii)$ of \cref{ass:cellular_th_wms_from_small} shows that the latching map is in $J$.
\end{proof}

\begin{prop}\label{prop:ReedyDiag_satisfiesMainTh} Under \cref{ass:cellular_th_wms_from_small}, the category $\Ccal^P$ of $P$-diagrams also satisfies similar conditions:

\begin{enumerate}[label=(\roman*)]

\item The class of $J$-Reedy morphisms contain isomorphisms, is stable under pushout and $\kappa$-small transfinite composition. 

\item Let $\Xcal \overset{j}{\cto} \Ycal \overset{i}{\cto} \Zcal$ be two composable $I$-Reedy morphisms with $\Xcal$ a $I$-Reedy diagram. If $j$ and $ i \circ j$ are both $J$-Reedy morphisms, then so does $i$.

\item For any $i: \Xcal \cto \Ycal$ a $I$-Reedy morphisms and $\Xcal$ a $I$-Reedy diagram, then there is a square of diagrams:

\[ 
\begin{tikzcd}
  \Xcal \ar[r,cto,"i"] \ar[dr,phantom,"\ulcorner"{name=M,description}]\ar[d,cto] & \Ycal \ar[d,cto] \\
\Ycal' \ar[r,cto] & \Zcal \ar[from = M,cto]\\
\end{tikzcd}
\]

where all maps, as well as $\Ycal \coprod_\Xcal \Ycal' \cto \Zcal$ are $I$-Reedy morphisms, and both $\Ycal \to \Zcal$ and $\Ycal' \to \Zcal$ are $J$-Reedy morphisms.
\end{enumerate}

If $\Ccal$ is further assumed to satisfies the conditions $(iii')$ of \cref{ass:cellular_th_wms_from_small}, then, in condition $(iii)$ above, we can assume that $\Ycal \to \Zcal$ admits a retraction.
\end{prop}

\begin{proof}
Condition $(i)$ is easy and Condition $(ii)$ follows immediately from \cref{lem:charac_Reedy_J_morphisms}.
We prove condition $(iii)$. Let $i:\Xcal \cto \Ycal$ be a $I$-Reedy morphism, with $\Xcal$ a $I$-Reedy diagram. We construct $\Ycal'$ and $\Zcal$ satisfying the requirement of condition $(iii)$ by induction on $p \in P$. We assume that $\Ycal'$ and $\Zcal'$ are already constructed for all $p' <p$. Taking the colimit over $\partial p$ in $\Ccal$ gives

\[
\begin{tikzcd}
  \Xcal(\partial p) \ar[r,cto] \ar[d,cto]  \ar[dr,phantom,"\ulcorner"{name=M,description}] & \Ycal(\partial p) \ar[d,cto,"\in J"] \\
 \Ycal'(\partial p ) \ar[r,cto,"\in J"swap] & \Zcal(\partial p) \ar[from = M,cto] 
\end{tikzcd}
\]

One can take the pushout along $\Xcal(\partial p) \rightarrow \Xcal(p)$ to get:

\[
\begin{tikzcd}
  \Xcal(p) \ar[r,cto] \ar[d,cto]  \ar[dr,phantom,"\ulcorner"{name=M,description}] &  \displaystyle \Ycal(\partial p) \coprod_{\Xcal(\partial p)} \Xcal(p) \ar[d,cto,"\in J"] \\
  \displaystyle \Ycal'(\partial p ) \coprod_{\Xcal(\partial p)} \Xcal(p) \ar[r,cto,"\in J"swap] &  \displaystyle \Zcal(\partial p) \coprod_{\Xcal(\partial p)} \Xcal(p) \ar[from = M,cto] 
\end{tikzcd}
\]

One then adds $\Ycal(p)$ to this diagram, and take a pushout:

\[
\begin{tikzcd}
  \Xcal(p) \ar[r,cto] \ar[d,cto]  \ar[dr,phantom,"\ulcorner"{name=M,description}] &  \displaystyle \Ycal(\partial p) \coprod_{\Xcal(\partial p)} \Xcal(p) \ar[dr,phantom,"\ulcorner"{very near end}] \ar[d,cto,"\in J"] \ar[r,cto] & \Ycal(p) \ar[d,cto,"\in J"] \\
  \displaystyle \Ycal'(\partial p) \coprod_{\Xcal(\partial p)} \Xcal(p) \ar[r,cto,"\in J"] &  \displaystyle \Zcal(\partial p) \coprod_{\Xcal(\partial p)} \Xcal(p) \ar[from = M,cto] \ar[r,cto] & W
\end{tikzcd}
\]

One then applies condition $(iii)$ to the bottom line of this diagram to get:

\[
\begin{tikzcd}
  \Xcal(p) \ar[r,cto] \ar[d,cto]  \ar[dr,phantom,"\ulcorner"{name=M,description}] &  \displaystyle \Ycal(\partial p) \coprod_{\Xcal(\partial p)} \Xcal(p) \ar[dr,phantom,"\ulcorner"{very near end}] \ar[d,cto,"\in J"] \ar[r,cto] & \Ycal(p) \ar[d,cto,"\in J"] \\
  \displaystyle \Ycal'(\partial p) \ar[d,cto] \coprod_{\Xcal(\partial p)} \Xcal(p) \ar[drr,phantom,"\ulcorner"{description, name = M2}] \ar[r,cto,"\in J"] &  \displaystyle \Zcal(\partial p) \coprod_{\Xcal(\partial p)} \Xcal(p) \ar[from = M,cto] \ar[r,cto] & W  \ar[d,cto,"\in J"] \\
\Ycal'(p) \ar[rr,cto,"\in J"] & & \Zcal(p) \ar[from=M2,cto]
\end{tikzcd}
\]

Where we have already named the new objects $\Ycal'(p)$ and $\Zcal(p)$, because these are the one we use to extend $\Ycal'$ and $\Zcal$ to $p$. The maps $\Ycal(p) \to \Zcal(p)$ and $\Ycal'(p) \to \Zcal(p)$ are in $J$ by construction, so it follows from \cref{lem:charac_Reedy_J_morphisms} and our induction hypothesis that $\Ycal \to \Zcal$ and $\Ycal' \to \Zcal$ are $J$-Reedy morphisms for this extension. In order to conclude one needs to check that with these definitions the morphism $\Ycal \coprod_{\Xcal} \Ycal' \rightarrow \Zcal$ is a $I$-Reedy morphism at $p$, i.e. that:

\[ \left( \left( \Ycal(p) \coprod_{\Xcal(p)} \Ycal'(p) \right) \coprod_{\left(\displaystyle  \Ycal(\partial p) \coprod_{\Xcal(\partial p)}\Ycal'(\partial p) \right)} \Zcal(\partial p) \right) \rightarrow \Zcal(p) \]

is in $I$. But this map is exactly the cellular map making the bottom rectangle a cellular square, indeed both these maps unfold to the canonical map from the colimits of the solid diagram below

\[
\begin{tikzcd}
\Xcal(\partial p) \arrow[dd] \arrow[rr] \arrow[rd] &                                          & \Ycal(\partial p) \arrow[rd] \arrow[dd] &                                      \\
                                                   & \Ycal'(\partial p) \arrow[rr,crossing over] &                                         & \Zcal(\partial p) \arrow[dd, dotted] \\
\Xcal(p) \arrow[rd] \arrow[rr]                     &                                          & \Ycal(p) \arrow[rd, dotted]             &                                      \\
                                                   & \Ycal'(p) \arrow[rr, dotted] \arrow[from=uu,crossing over]            &                                         & \Zcal(p)                            
\end{tikzcd}
\]

\noindent to $\Zcal(p)$.

If $\Ccal$ is further assumed to satisfy condition $(iii')$ of \cref{ass:cellular_th_wms_from_small}, then we also inductively assume $\Ycal(\partial p) \cto \Zcal(\partial p)$ has a retraction, and, condition $(iii')$ of \cref{ass:cellular_th_wms_from_small} gives a retraction to the morphism $W \cto \Zcal(p)$. Our two retractions can then be inserted in the diagram:

\[
\begin{tikzcd}
  \Xcal(p) \ar[r,cto] \ar[d,cto]  \ar[dr,phantom,"\ulcorner"{name=M,description}] &  \displaystyle \Ycal(\partial p) \coprod_{\Xcal(\partial p)} \Xcal(p) \ar[dr,phantom,"\ulcorner"{very near end}] \ar[d,cto] \ar[r,cto] & \Ycal(p) \ar[d,cto] \\
  \displaystyle \Ycal'(\partial p) \ar[d,cto] \coprod_{\Xcal(\partial p)} \Xcal(p) \ar[drr,phantom,"\ulcorner"{description, name = M2}] \ar[r,cto] &  \displaystyle \Zcal(\partial p) \coprod_{\Xcal(\partial p)} \Xcal(p) \ar[from = M,cto] \ar[r,cto] \ar[u,bend right=30] & W  \ar[d,cto] \ar[u,bend right=30] \\
\Ycal'(p) \ar[rr,cto] & & \Zcal(p) \ar[u,bend right=30] \ar[from=M2,cto]
\end{tikzcd}
\]

\noindent where we have used that a pushout of a map with a retraction has a (canonical) retraction as well. The composite of the two retractions of the right vertical arrows produce a retraction that is compatible with the one we started with on $\Ycal(\partial p) \cto  \Zcal(\partial p)$, which (inductively) produces a morphism in the category of diagram.
\end{proof}

\section{Model structure from small objects}
\label{sec:MS_from_small}

This section is devoted to the proof of \cref{th:wms_from_small,th:left_semi_ms_from_small}.

At this point we work under the assumption of \cref{th:wms_from_small} modified as suggested in \cref{rk:cellular_are_enough}, that is the \cref{ass:cellular_th_wms_from_small} of the previous section. We in particular have a pre-model structure structure on $\Ccal$: we call fibrations and anodyne fibrations the map with the lifting property against $J$ and $I$ respectively. The cofibrations are the retract of $I$-cellular maps and the anodyne cofibrations are the retract of $J$-cellular maps.

\begin{theorem}\label{th:wms_from_small_cellular} Let $\Ccal$ be as in \cref{ass:cellular_th_wms_from_small}. Then there is a weak model structure\footnote{In the sense of \cite{henry2018weakmodel}, see \cref{rk:More_general_wms} of the present article.} on $\Ccal$ whose ``cofibrations'' are the cellular maps and whose fibrations are the map with the right lifting property against all maps in $J$.\end{theorem}

\begin{proof} By the ordinary small object argument any map can be factored as a $I$-cellular map followed by an anodyne, hence acyclic, fibration. Similarly, any arrow can be factored as a $J$-cellular map followed by a fibration and $J$-cellular maps are acyclic cofibrations. This shows that $\Ccal$ satisfies the factorization axioms of Definition 2.1.10 of \cite{henry2018weakmodel}. It only remains to show the existence of relative cylinder and path objects.

We then show that a $I$-cellular map $i: A\cto B$ between $I$-cellular objects has a weak relative cylinder object. By \cref{lem:FatSOA}, $A = \Acal(P)$ for $P$ a locally $\kappa$-small well-founded diagram and $\Acal: P \to \Ccal$ a $I$-Reedy diagram, and (applying \ref{lem:FatSOA} again) $B = \Bcal(Q)$ for $Q$ an extension of $P$, also locally $\kappa$-small and $\Acal$ is the restriction of $\Bcal$ to $P$.

Now, by point $(iii)$ of \cref{prop:ReedyDiag_satisfiesMainTh} we have in the category of $Q$-diagram a square:

\[ 
\begin{tikzcd}
  \Acal \ar[r,cto,"i"] \ar[dr,phantom,"\ulcorner"{name=M,description}]\ar[d,cto] & \Bcal \ar[d,cto] \\
\Bcal' \ar[r,cto] & \Ccal \ar[from = M,cto]\\
\end{tikzcd}
\]

Where all maps (including $\Bcal' \coprod_\Acal \Bcal \to \Ccal$) are $I$-Reedy maps, and the two maps $\Bcal \to \Ccal$ and $\Bcal' \to \Ccal$ are $J$-Reedy. Taking the colimits on $Q$, gives a similar diagram in $\Ccal$

\[ 
\begin{tikzcd}
  A \ar[r,cto,"i"] \ar[dr,phantom,"\ulcorner"{name=M,description}]\ar[d,cto] & B \ar[d,cto,"\in J\text{-cell}"] \\
B' \ar[r,cto,"\in J\text{-cell}"swap] & C \ar[from = M,cto]\\
\end{tikzcd}
\]

By point \ref{item:self_composed_span} in \cref{sec:wms}, this shows that $A \cto B$ admits a relative weak cylinder object, given by $B \coprod_A B \cto C \coprod_B' C \to C$.

To conclude that $\Ccal$ is a weak model category, we will use the dual of Proposition~2.3.2.(ii) of \cite{henry2018weakmodel}, and show that for $A$ a bifibrant object in $\Ccal$, if we have a factorization $A \overset{j}{\hookrightarrow} B \overset{p}{\twoheadrightarrow} A$ with $j$ a $J$-cellular map and $p$ a fibration, then $p$ is an acyclic fibration. By proposition~2.3.2 of \cite{henry2018weakmodel}, this will imply the existence of relative path objects for every fibrations and concludes the proof of the existence of the weak model structure.

In order to do that, one considers a further factorization:
\[\begin{tikzcd}
A \ar[r,cto,"j","\in J"{swap}] & B \ar[dr,cto,"i"] \ar[rr,two heads,"p"] & & A\\
& & C \ar[ur,->>,"q","\sim"{swap}]  
\end{tikzcd}\]
where $i$ is $I$-cellular. The maps $A \overset{\in J}{\cto} B \cto C$ can be lifted as above, using three times lemma \ref{lem:FatSOA}, to maps of $P$-diagram, for $P$ a locally $\kappa$-small well-founded poset,

\[ \emptyset  \overset{I\text{-Reedy}}{\cto} \Acal \overset{J\text{-Reedy}}{\cto} \Bcal \overset{I\text{-Reedy}}{\cto} \Ccal \]

By condition $(iii)$ of \cref{prop:ReedyDiag_satisfiesMainTh}, combined with point \ref{item:self_composed_span} of \cref{sec:wms}, one can construct a relative weak cylinder object for $\Acal \cto \Ccal$ in the category of Reedy diagram. By that we mean a diagram%
\[
\begin{tikzcd}
\displaystyle  \Ccal \coprod_\Acal \Ccal \ar[d] \ar[r,cto] & I_\Acal \Ccal \ar[d] \\
\Ccal \ar[r,cto] & D_\Acal \Ccal
\end{tikzcd}
\]
where both maps $\Ccal \cto I_\Acal \Ccal$ and the map $\Ccal \to D_\Acal \Ccal$ are $J$-Reedy, $\Ccal \coprod_\Acal \Ccal \cto I_\Acal \Ccal$ is $I$-Reedy. By condition $(ii)$ of \ref{prop:ReedyDiag_satisfiesMainTh} the composite morphisms%
\[
\begin{tikzcd}
 \Ccal \coprod_{\Acal} \Bcal \ar[r,cto] & \Ccal \coprod_{\Acal} \Ccal \ar[r,cto] & I_{\Acal} \Ccal
\end{tikzcd}
\]%
is a $J$-Reedy morphism, because the maps $\Ccal \cto  \Ccal \coprod_{\Acal} \Bcal $ and $\Ccal \cto  I_{\Acal} \Ccal$ are both $J$-Reedy. Taking the colimits of these in $\Ccal$ one obtains a relative weak cylinder $I_{A} C$ in $\Ccal$ with the property that $C \coprod_A B \cto I_A C$ is $J$-cellular. One then first form a lift in the square:

\[
\begin{tikzcd}
C \coprod_A B \ar[r,"(jq\,Id)"] \ar[d,cto,"J\text{-cell}"swap] & B \ar[d,->>,"p"] \\
I_A C \ar[ur,dotted,"h"] \ar[r,"r_q"] & A 
\end{tikzcd}
\]

Where $r_q$ denotes any self homotopy of $q:C \rightarrow A$ (these exists by using the lifting property of $C \overset{J\text{-cell}}{\cto} D_A C$ against $A$). One denotes by $h_1$ the composite of $h$, with second map $C \rightarrow I_A C$ (i.e. not the one used in the left vertical arrow above). By construction one has $h_1 \circ i =Id_B$. One then form a lifting in the square:

\[
\begin{tikzcd}
  C \coprod_B C \ar[d,cto] \ar[r,"(i h_1 \, Id)"]& C \ar[d,->>,"\sim","q"{swap}] \\
Id_B C \ar[r,"r_q"] \ar[ur,dotted,"\eta"] & A \\
\end{tikzcd}
\]

Finally these two maps $h_1$ and $\eta$ fits together in a retract diagram:

\[
\begin{tikzcd}
B \ar[d,cto,"i"] \ar[rr, bend left=30,equal] \ar[r,cto,"i"] & C \ar[r,"h_1"] \ar[d,cto,"J\text{-cell}","e_0"{swap}] & B \ar[d,cto,"i"] \\
C \ar[r,cto,"e_1"] \ar[rr, bend right=30,equal] & Id_B C \ar[r,"\eta"] & C \\
\end{tikzcd}
\]

This shows that $i:B \cto C$ is a retract of a $J$-cellular morphism, in particular it has the left lifting property against $p$, hence by the usual retract lemma, $p$ is a retract of $q$, hence is an anodyne fibration which concludes the proof. \end{proof}

Now the restrictions in \cref{th:wms_from_small_cellular} that the ``cofibrations'' of the weak model structure are only the $I$-cellular maps and not all the actual cofibrations is easily lifted:

\begin{cor}\label{cor:wms_from_small_2} Let $\Ccal$ be as in \cref{ass:cellular_th_wms_from_small}. Then its induced pre-model structure is a weak model structure. This weak model structure is Quillen equivalent, and has the same weak equivalences when it makes sense, as the one from \cref{th:wms_from_small_cellular}.
\end{cor}

\Cref{th:wms_from_small} corresponds exactly to the special case of this corollary where the class $I$ is also closed under retract in the first place.

\begin{proof}
This puts on $\Ccal$ the same notion of fibrations and acyclic fibration as the weak model structure from \ref{th:wms_from_small_cellular}. It follows that every fibration between fibrant objects has a relative path object and that acyclic fibrations satisfies $2$-out-of-$3$ amongst fibrations. This is enough to get a weak model structure on $\Ccal$ by Proposition~2.3.3 of \cite{henry2018weakmodel}.

Both model structure have the same fibrations and acyclic fibrations, so they have the same equivalence between fibrant objects. Fibrant replacement for the model structure of \ref{th:wms_from_small_cellular} are still fibrant replacement for the present model structure, and so an arrow between object that are fibrant or cofibrant in the model structure of \ref{th:wms_from_small_cellular} is an equivalence if and only if it is an equivalence in the sense of the present model structure. \end{proof}

We next move to the proof of \cref{th:left_semi_ms_from_small}

\begin{prop}\label{prop:iii'_imply_cylinder} Assume that $\Ccal$ satisfies \cref{ass:cellular_th_wms_from_small}, including condition $(iii')$, then:

\begin{itemize}

\item Every cofibrant objects of the weak model structures of \cref{th:wms_from_small_cellular} or \cref{cor:wms_from_small_2} have a strong cylinder object, i.e. a factorization of the codiagonal $X \coprod X \cto IX \overset{\sim}{\to} X$ in a cofibration followed by an equivalence.

\item Every anodyne fibration between objects that are fibrants or cofibrants is a weak equivalence.

\end{itemize}
\end{prop}

\begin{proof} In the proof of \cref{th:wms_from_small_cellular}, if we further assume condition $(iii')$, then for each cellular map $A \cto B$ between cellular objects in $\Ccal$ we have a diagram%
\[ \begin{tikzcd}
  A \ar[r,cto] \ar[dr,phantom,"\ulcorner"{name=M,description}]\ar[d,cto] & B \ar[d,cto,"J\text{-cell}"] \\
B' \ar[r,cto,"J\text{-cell}"swap] & C \ar[from = M,cto] \ar[u,bend left=30,dotted,"r"]
\end{tikzcd} \]
with a retraction $r:C \to B$ of the map $B \cto C$. It follows that the weak relative cylinder
\[
\begin{tikzcd}
\displaystyle B \coprod_A B \ar[r,cto] \ar[d] & \displaystyle I_A B = C \coprod_{B'} C \ar[d]  \\
B \ar[r] & C \ar[l,bend left=30,"r"]
\end{tikzcd}
\]
constructed by point \ref{item:self_composed_span} of \cref{sec:wms}, is actually a relative strong cylinder.

This shows the first point for cellular objects, i.e. for the weak model structure of \cref{th:wms_from_small_cellular}. This is enough to show that every acyclic fibration $p: X \overset{\approx}{\twoheadrightarrow} Y$ from a cellular object to a cofibrant object is a weak equivalence. Indeed such a $p$ has a section $s:Y \to X$ because $Y$ is cofibrant, and a lifting in the square:
\[\begin{tikzcd}
 X \coprod X \ar[d,cto] \ar[rr,"(sp{,}\text{Id})"] & & X \ar[d,two heads,"p"] \\
I X \ar[r] \ar[urr,dotted] & X \ar[r,"p"] & Y
\end{tikzcd}
\]
produces an homotopy between $sp$ and the identity, hence showing that $s$ and $p$ are inverse of each other in the homotopy category.

Acyclic fibrations with fibrant targets are equivalences because $\Ccal$ is a weak model category (see for example Proposition 2.2.3 of \cite{henry2018weakmodel}). Any acyclic cofibration either between cofibrant objects or from a fibrant object to a cofibrant object is also an equivalence: indeed in both case we can precompose it with an acyclic fibrations $p$ with cellular domain, both $p$ and the composite are anodyne fibration with cellular domain and either fibrant or cofibrant targets, so are equivalences by the two observations above. Hence by $2$-out-of-$3$ the anodyne fibration we started from is also an equivalence. This concludes the proof of the second point.

The first point for general cofibrant objects follows by considering a (cofibration,anodyne fibration) factorization $X \coprod X \cto IX \overset{\ano}{\twoheadrightarrow} X$. The two maps $X \to IX$ are automatically acyclic cofibrations by $2$-out-of-$3$ for weak equivalences.
\end{proof}

\begin{prop}\label{prop:left_semi_from_small_cellular} Let $\Ccal$ be a $\lambda$-cellular category with a set $J$ of maps satisfying all the conditions of \cref{ass:cellular_th_wms_from_small}, including $(iii')$. Then the core left saturation\footnote{In the sense of section 4 of \cite{henry2019CombWMS}.} of the pre-model structure on $\Ccal$ generated by $I$ and $J$ is a Fresse left semi-model category.

It has the same cofibrations as $\Ccal$ and its fibrant objects and fibrations between fibrant objects are characterized as in \cref{th:left_semi_ms_from_small} by the lifting property against all maps in $J$.
\end{prop}

Here again, \cref{th:left_semi_ms_from_small} corresponds to the special case where the class $I$ is assumed to be closed under retract.

\begin{proof} As observed above, under these assumptions $\Ccal$ with the factorization system generated by its cellular maps and by $J$ is a combinatorial pre-model category which:

\begin{itemize}

\item Is a weak model category by \cref{th:wms_from_small_cellular} and \cref{cor:wms_from_small_2}.

\item In which every cofibrant object has a strong cylinder object by \cref{prop:iii'_imply_cylinder}.
  
\end{itemize}

It follows from Section 4 of \cite{henry2019CombWMS} that its core left saturation exists (because it is combinatorial). It has the same cofibrations and core fibrations, so its fibrant objects and core fibrations are indeed described by the lifting property against $J$, and the two conditions above are immediately still satisfied as they only involve core cofibrations and fibrations. It hence follows by \cref{rk:charac_of_left_ms} (i.e. Theorem~3.7 of \cite{henry2019CombWMS}) that it is a Fresse left semi-model category.
\end{proof}

\section{Model structure from localizer}
\label{sec:localizer}

The goal of this section is to prove our main result, i.e. \cref{main_th:Quillen}, or rather its generalized version \cref{main_th:Quillen_refined}. One implication is fairly easy:

If $\Ccal$ is a combinatorial model category then its class of equivalences $\Wcal$ is a small-generated localizer: It is clearly a localizer and we claim that if $J$ is any set of generating anodyne cofibrations then $\Wcal(J) = \Wcal$. Indeed $J \subset \Wcal$ hence $\Wcal(J) \subset \Wcal$, and as any map in $\Wcal$ factors as an anodyne cofibration followed by an anodyne fibrations and both anodyne cofibrations and anodyne fibrations are in $\Wcal(J)$, this implies that $\Wcal  \subset \Wcal(J)$.

Also if $\Ccal$ is a combinatorial model category, then a map which is a fibration and has the lifting property against all core cofibrations is a weak equivalence hence an anodyne fibration, hence we always have a set $J_0$ as in \cref{main_th:Quillen_refined}: any set of generating anodyne cofibrations works. We only need to show the converse:

\begin{assumptions}\label{ass:general_main_th} In this section, we work under the assumptions of \cref{main_th:Quillen_refined}. More precisely, we consider a combinatorial structured category $\Ccal$ endowed with a small-generated $\Ccal$-localizer $\Wcal=\Wcal(S)$, for $S$ a set of arrows. We assume that we have a set $J_0$ of cofibrations in $\Ccal$ that are all in $\Wcal$ and such that an arrow of $\Ccal$ that has the right lifting property against all core cofibrations and all arrows in $J_0$ is an anodyne fibration.\end{assumptions}

We start with the following observation, that has nothing to do with small-generation and applies to any localizer:

\begin{lemma} \label{lem:2_out_of_6} For any $\Ccal$-localizer $\Wcal$,
\begin{itemize}
\item $\Wcal$ satisfies the $2$-out-of-$6$ condition, that is given $f,g,h$ composable arrows such that $g \circ f$ and $h \circ g$ are in $\Wcal$, then $f,g,h$ and $hgf$ are in $\Wcal$.
\item $\Wcal \cap \cof$ is closed under retract.
\end{itemize}

\end{lemma}

We are grateful to D.-C.~Cisinski for pointing this out. In a previous draft of the paper, we had to include the $2$-out-of-$6$ condition in the definition of a localizer, hence creating a mismatch with the previous work of Cisinski in the case of Grothendieck toposes from \cite{cisinski2002theories} and \cite{cisinski2006prefaisceaux}. \Cref{lem:2_out_of_6} can be deduced from Proposition~6.2 of \cite{cisinski2010categories}, which more generally shows that $\Wcal$ is strongly saturated, i.e. is exactly the class of arrows inverted in the localization at $\Wcal$ (in particular $\Wcal$ itself is also closed under retract). We give below a version of this argument more adapted to our setting, but it is directly inspired from the proof of Proposition~6.2 in \cite{cisinski2010categories}.

\begin{proof} For the first point, we start from a diagram

\[\begin{tikzcd}
A \arrow[rd, "f"'] \arrow[rr, "\in \Wcal"] &                                       & C \arrow[rd, "h"] &   \\
                                      & B \arrow[rr, "\in \Wcal"'] \arrow[ru, "g"] &                   & D.
\end{tikzcd}\]
  
As $\Wcal$ satisfies $2$-out-of-$3$, it is enough to show that $f \in \Wcal$ to conclude that first $g$, then $h$ and finally $hgf$ are in $\Wcal$. Moreover, by replacing $f$ by a (cofibration, anodyne fibration) factorization we can freely assume that $f$ is a cofibration. We then gradually extend this diagram as follow, first we factor the map $g: B \to C$ into a cofibration followed by an anodyne fibration $B \cto A_1 \overset{\ano}{\twoheadrightarrow} C$ to get
\[\begin{tikzcd}
A \arrow[rd, "f"', cto] \arrow[rr, "\in \Wcal", cto] &                                             & A_1 \arrow[r, "\ano", two heads] & C \arrow[d] \\
                                                       & B \arrow[rr, "\in \Wcal"'] \arrow[ru, cto] &                                     & D          
\end{tikzcd}\]

The map $A \cto A_1$ is a cofibration as a composite of cofibrations and is in $\Wcal$ by $2$-out-of-$3$ as both $A_1 \overset{\ano}{\twoheadrightarrow} C$ and $A \to C$ are in $\Wcal$. We then factor the map $A_1 \to D$ in $A_1 \cto B_1 \overset{\ano}{\twoheadrightarrow} D$ to get
\[\begin{tikzcd}
A \arrow[rd, "f"', cto] \arrow[rr, "\in \Wcal", cto] &                                                   & A_1 \arrow[rd, cto] \arrow[rr, "\ano", two heads] &                                     & C \arrow[d, "h"] \\
                                                       & B \arrow[rr, "\in \Wcal"', cto] \arrow[ru, cto] &                                                       & B_1 \arrow[r, "\ano", two heads] & D               
\end{tikzcd}\]
Here again $B \cto B_1$ is a  cofibration by composition, the map $B \to D \in \Wcal$ from the previous diagram $B_1 \to D$ is an anodyne fibration, hence $B \cto B_1$ is in $\Wcal$. We iterate this to gradually construct a diagram
\[ \begin{tikzcd}
A \arrow[rd, "f"', cto] \arrow[rr, "\in \Wcal", cto] &                                                   & A_1 \arrow[rd, cto] \arrow[rr, "\in \Wcal", cto] &                                                     & A_2 \arrow[rd, cto] \arrow[rr, "\in \Wcal", cto] &                                                    & \dots \arrow[r, "\ano", two heads]  & C \arrow[d, "h"] \\
                                                       & B \arrow[rr, "\in \Wcal"', cto] \arrow[ru, cto] &                                                    & B_1 \arrow[ru, cto] \arrow[rr, "\in \Wcal"', cto] &                                                    & B_2 \arrow[r, "\in \Wcal"', cto] \arrow[ru, cto] & \dots \arrow[r, "\ano"', two heads] & D               
\end{tikzcd} \]

at each step $A_i$ is constructed as a factorization $B_{i-1} \cto A_i \twoheadrightarrow C$, the map $A_{i-1} \cto A_i$ is a cofibration by composition, $A \cto A_i$ is in $\Wcal$ by $2$-out-of-$3$ for the composite $A \cto A_i \overset{\ano}{\twoheadrightarrow} C$, hence $A_{i-1} \cto A_i$ is in $\Wcal$ by $2$-out-of-$3$ in $A \cto A_{i-1} \cto A_i$. We then construct $B_i$ as a factorization $A_i \cto B_i \twoheadrightarrow D$ and conclude by similar arguments.

The colimit of the diagonal map in the diagram above induces an isomorphism $\colim B_i \simeq \colim A_i$, and that both maps $A \cto \colim A_i$ and $B \cto \colim B_i$ are in $\Wcal$ as $\omega$-transfinite compositions of cofibrations in $\Wcal$. Hence $f \in \Wcal$ by $2$-out-of-$3$.

For the second point, as $\Wcal \cap \cof$ is closed under pushouts, if $i$ is a retract of $j_0 \in \Wcal \cap \cof$, then $i$ is also a retract of $j$, the pushout of $j_0$ to the domain of $i$. We hence have an ``identity on domain'' retract diagram as on the left below
\[
\begin{tikzcd}
A \ar[r,equal] \ar[d,"i",cto] & A \ar[d,cto,"\in \Wcal","j"swap] \ar[r,equal] & A \ar[d,cto,"i"] \\ 
B \ar[rr,bend right=30,equal] \ar[r,"u"] & C \ar[r,"v"] & B 
\end{tikzcd} \qquad 
\begin{tikzcd}
A \ar[d,cto,"i"swap] \ar[r,"j \in \Wcal",cto] & C \ar[d,"v"] \\
 B \ar[ur,"u"] \ar[r,equal] & B 
\end{tikzcd}
\]
which can be rearranged in the $2$-out-of-$6$ diagram on the right above, hence showing that $i \in \Wcal$. As $i$ is a retract of a cofibration, it is also a cofibration hence $i \in \Wcal \cap \cof$.
\end{proof}

\begin{remark}\label{rk:SOA_access_rank} The (cofibrations, anodyne fibrations) weak factorization system produced by the small object argument is functorial, i.e. take the form of a functor $\Ccal^\to \to \Ccal^{\to \to}$. A quick analysis shows that if $\Ccal$ is $\lambda$-combinatorial then this functor preserves $\lambda$-filtered colimits (see for example the proof of Proposition 4.22 in \cite{garner2009understanding}, or Proposition B.5 in \cite{henry2019CombWMS}). That is the factorization functor is accessible, it hence follows that there is a regular cardinal $\kappa$ such that it preserves $\kappa$-presentable objects, i.e. it induces a functor
\[ (\Pr_\kappa \Ccal)^\to =  \Pr_\kappa (\Ccal^\to) \to \Pr_\kappa(\Ccal^{\to\to})=(\Pr_\kappa \Ccal)^{\to\to}\]
Hence showing that one can construct (cofibration, anodyne fibration) factorizations within $\Pr_\kappa \Ccal$. An explicit description of the cardinals $\kappa$ for which this holds is given in Theorem~4.4 of \cite{low2016heart}, what we need is a regular cardinal $\kappa$ such that
\begin{itemize}
\item For any two $\lambda$-presentable objects $X$ and $Y$, the set of all functions from $X$ to $Y$ is $\kappa$-small.
\item The set of generating cofibrations is $\kappa$-small.
\item $\lambda$ is sharply below $\kappa$, i.e.  $\lambda < \kappa$ and given any $\kappa$-small set $X$ , the poset $\Pcal_\lambda(X)$ of $\lambda$-small subsets of $X$ admit a $\kappa$-small cofinal subsets. An interested reader can consult section 1 of \cite{low2016heart} for some introduction to this notion. We recall that $\lambda$ is always sharply below $\lambda^+$ and if $\lambda \leqslant \kappa$ then $\lambda$ is sharply below $(2^\kappa)^+$. In particular one can find arbitrary large $\kappa$ with this property.
\end{itemize}

\end{remark}

\begin{assumptions}\label{ass:on_kappa} We fix $\kappa$ a regular cardinal such that:
\begin{itemize}
\item $\kappa$ is uncountable.
\item $\Pr_\kappa \Ccal$ admits a (cofibration, anodyne fibration) factorization. See \cref{rk:SOA_access_rank} for details of what this means for $\kappa$.
\item $ S \subset \Pr_\kappa \Ccal^\to$.
\item $J_0 \subset \Pr_\kappa \Ccal^\to$. That is, an arrow that has the lifting property against all cofibrations in $\Wcal$ between $\kappa$-presentable objects and all core cofibrations is an anodyne fibration.
\end{itemize}

Given such a $\kappa$, we take $I$ to be the set of cofibrations between $\kappa$-presentable objects, and $J$ to be the subset of $I$ of these cofibrations that are in $\Wcal$. These sets $I$ and $J$ are respectively our sets of generating cofibrations and anodyne cofibrations of a pre-model structure on $\Ccal$, which we will show is a Quillen model structure. It follows from \cref{rk:cellular_are_enough} (and the second point of \cref{lem:2_out_of_6} in the case of $J$) that $I$-cellular maps and $J$-cellular maps coincide with $I$-cofibrations and $J$-cofibrations. \end{assumptions}

\begin{lemma}\label{lem:left_sat} The pre-model structure on $\Ccal$ generated by $I$ and $J$ as in \cref{ass:on_kappa} is saturated, that is any $I$-cofibration which is an acyclic cofibration is a $J$-cofibration and any $J$-fibration which is an acyclic fibration is a $I$-fibration. \end{lemma}

\begin{proof} Right saturation follows from the existence of the set $J_0$ (see \cref{ass:general_main_th}) and the last condition of \cref{ass:on_kappa} which imply that $J_0 \subset J$. Hence a $J$-fibration with the lifting property against all core cofibrations is an anodyne fibration.

We move to left saturation. By Theorem~4.1 of \cite{henry2019CombWMS}, as $\kappa$ is an uncountable regular cardinal, the left saturation of the $\kappa$-combinatorial pre-model category $\Ccal$, (whose anodyne cofibrations are the acyclic cofibrations of $\Ccal$) is also $\kappa$-combinatorial. So it is enough to show that all acyclic cofibrations between $\kappa$-presentable objects are anodyne cofibrations to conclude the proof.

Let $A \cto B$ be such an acyclic cofibration between $\kappa$-presentable objects. We first take a fibrant replacement of $B$
\[ A \cto B \overset{\ano}{\cto} D \]
and a ($J$-cofibration, fibration) factorization $A \overset{\ano}{\cto} C \twoheadrightarrow D$ of the composite map $A \to D$. This gives us a solid diagram
\[
\begin{tikzcd}
A \ar[d,cto] \ar[r,"\ano",cto] & C \ar[d,two heads,"\pi"] \\
B \ar[ur,dotted] \ar[r,"\ano"swap,cto] & D
\end{tikzcd}
\]
and a dotted arrow exists because $\pi$ is a core fibration and $A \cto B$ is assumed to be an acyclic cofibration. The two $J$-cofibrations are in $\Wcal$ by assumptions, so as $\Wcal$ satisfies $2$-out-of-$6$ by \cref{lem:2_out_of_6}, the map $A \cto B$ is in $\Wcal$. As it is a map between $\kappa$-presentable objects, it is in $J$, and hence it is an anodyne cofibration. \end{proof}

\begin{prop}\label{prop:C_is_a_left_semi} The pre-model structure on $\Ccal$ generated by $I$ and $J$ is a Fresse left semi-model category. \end{prop}

In fact, as we already proved in \cref{lem:left_sat} that this premodel structure is right saturated, so it is even a Spitzweck left semi-model category.

\begin{proof} We apply \cref{th:left_semi_ms_from_small}, or rather its more precise form \cref{prop:left_semi_from_small_cellular}. As by \cref{lem:left_sat}, $\Ccal$ is already left saturated we do not need to take the core left saturation and we directly get that $\Ccal$ is a left semi-model category. Condition $(i)$ and $(ii)$ of \ref{ass:cellular_th_wms_from_small} are immediate by definition of $J$.

The only non-trivial condition is $(iii')$: given $A \cto B$ a cofibration between $\kappa$-presentable cofibrant objects in $\Ccal$ we form a (cofibration, anodyne fibration) factorization of the codiagonal map: (we can do this because of the second point of \cref{ass:on_kappa})
\[ B \coprod_A B \cto I_A B \overset{\ano}{\twoheadrightarrow} B \]
it is immediate by $2$-out-of-$3$ that $B \cto I_A B$ is in $\Wcal$, hence in $J$. Because of the second point of \cref{ass:on_kappa}, we can assume that $I_A B$ is also $\kappa$-presentable. This gives a diagram:
\[ \begin{tikzcd}
  A \ar[r,cto] \ar[dr,phantom,"\ulcorner"{name=M,description}]\ar[d,cto] & B \ar[d,cto] \\
B \ar[r,cto] & I_A B \ar[from = M,cto] \ar[u,two heads,"\ano"description,bend right=30]
\end{tikzcd} \]

as required by $(iii')$.

\end{proof}

In particular, $\Ccal$ being a left semi-model category it has a class of weak equivalences, characterized in terms of its weak factorization system. At this point, even if they are closely related, it is unclear that this class of equivalences coincide with $\Wcal$. When we talk about equivalences in $\Ccal$, we mean in the sense of this left semi-model category. Explicitly, an arrow with cofibrant domain is an equivalence if it factors as an anodyne cofibration (a $J$-cofibration) followed by an anodyne fibration and a general arrow is an equivalence if and only its pre-composition with a cofibrant replacement of its domain is an equivalence in the previous sense.

The following is a small improvement of proposition 7.3 of \cite{dugger2001combinatorial}, the proof is essentially the same.

\begin{lemma}\label{lem:kappa_filtered_colim_of_equivalence} In $\Ccal$, and more generally in any $\kappa$-combinatorial\footnote{In the sense of \cite{henry2019CombWMS}. It means that the underlying category is locally $\kappa$-presentable and the generating cofibrations and anodyne cofibrations are arrows between $\kappa$-presentable objects.} Fresse left semi-model category the class of weak equivalences is closed under $\kappa$-filtered colimits in $\Ccal^\to$.\end{lemma}

\begin{proof}
  Let $(f_i:X_i \to Y_i)_{i \in I}$ a $\kappa$-filtered diagram in $\Ccal^\to$ such that each $f_i:X_i \to Y_i$ is a weak equivalence. We denote by $f_\infty:X_\infty \to Y_\infty$ its colimits, which we want to prove is also a weak equivalence.

First we assume that all object involved ($X_i,Y_i,X_\infty,Y_\infty$) are fibrant. Then we can use the criterion that a map between fibrant objects is a weak equivalence if and only if it has the weak, or ``up to homotopy'' lifting property against a set of generating cofibrations (this is proved for weak model categories in Appendix~A.2 of \cite{henry2018weakmodel}, see Theorem~A.2.6 and Remark~A.2.7). Let $A \cto B$ be a generating cofibration, which can hence assume to be between $\kappa$-presentable objects, and consider a lifting problem as on the left below
\[  \begin{tikzcd}
 A \ar[d,cto] \ar[r]  & X_\infty \ar[d,"f_\infty"] \\
B \ar[r] & Y_\infty .
\end{tikzcd}
\qquad  \begin{tikzcd}
 A \ar[d,cto] \ar[r]& X_i \ar[d,"f_i"] \ar[r]  & X_\infty \ar[d,"f_\infty"] \\
B \ar[r] & Y_i \ar[r]& Y_\infty
\end{tikzcd} \]
As $A$ and $B$ are $\kappa$-presentable and the colimit over $I$ is $\kappa$-filtered, this lifting problem can be factored through one of the $f_i$ as on the right above.  Finally, as each $f_i$ is an equivalence between fibrant objects the square on the left has a weak solution, which gives a weak solution to the initial square

\[\begin{tikzcd}[row sep = 5]
A \arrow[ddd] \arrow[rr] \arrow[rd] &                        & X_i \arrow[r] \arrow[ddd,"f_i"] & X_\infty \arrow[ddd,"f_\infty"] \\
                                    & B \arrow[d] \arrow[ru] &                           &                      \\
                                    & I_A B \arrow[rd]       &                           &                      \\
B \arrow[rr] \arrow[ru]             &                        & Y_i \arrow[r]             & Y_\infty            
\end{tikzcd} \]
showing that $f_\infty$ is an equivalence.

In the general case, as the model structure is $\kappa$-combinatorial, the small object argument produces fibrant replacement $X \mapsto X^\fib$ and cofibrant replacement $X \mapsto X^\cof$ functors that are $\kappa$-accessible (as mentioned in \cref{rk:SOA_access_rank}, this can be deduced from the proof of Proposition 4.22 in \cite{garner2009understanding}, or from Proposition B.5 in \cite{henry2019CombWMS}). In a Fresse left semi-model structure, a map $f:X \to Y$ is an equivalence if and only if the corresponding map $(f^\cof)^\fib : (X^\cof)^\fib \to (Y^\cof)^\fib$ is an equivalence, so if each $f_i$ is an equivalence the $(f_i^\cof)^\fib$ are all equivalences as well, as the fibrant and cofibrant replacement preserve $\kappa$-filtered colimits their colimits is $(f_\infty^\cof)^\fib$ hence by the previous part of the proof it is an equivalence, and so $f_\infty$ is an equivalence.
\end{proof}

\begin{prop}\label{prop:J_cof_are_eq} All $J$-cofibrations are weak equivalences of $\Ccal$.\end{prop}

Note that it does not immediately follows from the fact that $J$-cofibration are anodyne cofibrations: as we only showed that $\Ccal$ is a left semi-model category, only the $J$-cofibrations with cofibrant domain are already known to be weak equivalences.

\begin{proof} First we observe that all maps in $J$ are weak equivalences. We recall that $J$ is exactly the set of cofibrations between $\kappa$-presentable objects that are in $\Wcal$.  Given $A \cto B$ such a cofibration, we can take a cofibrant replacement $A^\cof$ of $A$ and a (cofibration, anodyne fibration) factorization of $A^\cof \to B$ to get a diagram
\[
\begin{tikzcd}
  A^\cof \ar[r,cto] \ar[d,two heads,"\ano"] & B^\cof \ar[d,two heads,"\ano"] \\
A \ar[r,cto] & B
\end{tikzcd}
\]
where because of \cref{ass:on_kappa} we can assume that $A^\cof$ and $B^\cof$ are also $\kappa$-presentable.

By $2$-out-of-$3$ in $\Wcal$ the cofibration $A^\cof \cto B^\cof$ is in $\Wcal$ and between $\kappa$-presentable objects, so it is one of our generating anodyne cofibrations. In particular it is an anodyne cofibration with cofibrant domain, hence it is a weak equivalence. By $2$-out-of-$3$ for the weak equivalences of $\Ccal$ in the diagram above, it follows that $A \cto B$ is indeed a weak equivalence.

Next we prove the result for a general $J$-cofibration, or rather a $J$-cellular map. Let $A \overset{\ano}{\cto} B$ be a $J$-cellular map in $\Ccal$. Using \cref{lem:all_maps_are_cellular}, we chose a locally $\kappa$-small well-founded poset $P$ and a diagram $\Acal: P \to \Ccal$ such that $A= \colim_{p \in P} \Acal(p)$ where each $\Acal(p)$ is $\kappa$-presentable. By the fat small object argument, there is a locally $\kappa$-small extension $Q$ of $P$ and a diagram $\Bcal : Q \to \Ccal$ extending $\Acal$ so that the map $\Acal \to \Bcal$ is $J$-Reedy, or equivalently so that $\Bcal$ is $J$-Reedy at all $q \notin P$.

For each $U \sievei Q$ a $\kappa$-small sieve, the map $\Acal(P \cap U) \to \Bcal(U)$ is a $J$-cellular map between $\kappa$-presentable objects, hence is itself in $J$, in particular is an equivalence by the first half of the proof. Now the map $A \cto B$ is the colimit of all these maps over all $\kappa$-small sieve $U \sievei Q$, this colimit is $\kappa$-filtered, so it is an equivalence by \cref{lem:kappa_filtered_colim_of_equivalence}.\end{proof}

The following concludes the proof of \cref{main_th:Quillen,main_th:Quillen_refined}, which corresponds to the special case where we additionally assume that $I$ and $J$ are closed under retracts.

\begin{theorem}\label{th:final_Quillen} Under \cref{ass:general_main_th} and \cref{ass:on_kappa}, the model structure on $\Ccal$ generated by $I$ and $J$ is a Quillen model structure and its class of equivalence is $\Wcal(S)$. \end{theorem}

\begin{proof}
  We have showed as \cref{prop:C_is_a_left_semi} that it is a left semi-model structure, as \cref{lem:left_sat} that it is left saturated, and as \cref{prop:J_cof_are_eq} that all its anodyne cofibrations are weak equivalences. This immediately imply that it is a Quillen model structure: The only remaining axiom to prove is that a cofibration which is an equivalence is an anodyne cofibration, but given such a cofibration $i$, if we factor it as $p \circ j$ with $j$ an anodyne cofibration followed by a fibration $p$, then $p$ is an equivalence by $2$-out-of-$3$ hence is an anodyne fibration. It follows by the usual retract lemma that as $i$ has the lifting property against $p$ it is a retract of $j$ and hence is itself an anodyne fibration.

As $\Ccal$ is a Quillen model structure, the class of equivalences $\Ecal$ of $\Ccal$ is a localizer. We need to show that it is $\Wcal(S)$.

We first show that all maps in $S$ are equivalences of $\Ccal$, this proves that $\Wcal(S) \subset \Ecal$. Indeed if $f :a \to b$ is in $S$, then it is an arrow between $\kappa$-presentable objects by \cref{ass:on_kappa}. We consider $a \cto c \overset{\ano}{\twoheadrightarrow} b$ be a (cofibration, anodyne fibrations) factorization, where, again by \cref{ass:on_kappa} we have that $c$ is $\kappa$-presentable. The cofibration $a \cto c$ is in $\Wcal(S)$ by $2$-out-of-$3$ for $\Wcal(S)$, which imply that it is in $J$ and hence is an equivalence of $\Ccal$. By $2$-out-of-$3$ for $\Ecal$ this imply that $f$ is an equivalence.

Conversely, we show that $\Ecal \subset \Wcal(S)$. Indeed an arrow of $\Ecal$ factors as a $J$-cellular map followed by an anodyne fibrations. The anodyne fibration is in $\Wcal(S)$ by definition of a localizer, and as all maps in $J$ are cofibrations in $\Wcal(S)$, all $J$-cellular maps are in $\Wcal(S)$ again by definition of a localizer. This proves $\Ecal \subset \Wcal(S)$. \end{proof}

We conclude this section by a couple of lemmas that we initially thought would allow to remove the need for the additional assumption in \cref{main_th:Quillen_refined} (the existence of $J_0$) but in the end are not quite enough, see \cref{rk:last_condition}. We include this in the hope that someone will have an idea on how to remove the assumption:

\begin{lemma}\label{lem:acylic_fib_are_inW} Given any $\Ccal$-localizer $\Wcal$ on a structured $\Ccal$, a map with the right lifting property against all core cofibrations is in $\Wcal$. \end{lemma}

\begin{proof} Let $\pi:X \to Y$ be such a map. By taking (cofibration,anodyne fibrations) factorizations we form a solid diagram
\[\begin{tikzcd}
    X^\cof \ar[d,cto] \ar[r,two heads,"\ano"] & X \ar[d,"\pi"] \\
 Y^\cof \ar[ur,dotted] \ar[r,two heads,"\ano"] & Y
  \end{tikzcd} \]
The assumption on $f$ imply the existence of a dotted arrow. As the two anodyne fibrations are in $\Wcal$, the $2$-out-of-$6$ condition for $\Wcal$ imply that all maps in the diagram above, including $f$, are equivalences.
\end{proof}

\begin{lemma}\label{lem:W_and_fib_imply_trivial_fib} A map in $\Wcal$ is an anodyne fibration if and only if it has the right lifting property against all cofibration that are $\Wcal$. \end{lemma}

\begin{proof} An anodyne fibration has the lifting property against all cofibrations. So we only need to show that a map $\pi : X \to Y$ which is in $\Wcal$ and has the lifting property against all cofibrations in $\Wcal$ is an anodyne fibration. We consider a lifting problem

\[\begin{tikzcd}
  A \ar[d,cto] \ar[r] & X \ar[d,"\pi"] \\
 B \ar[r] & Y
\end{tikzcd}\]
  
\noindent against an arbitrary cofibration. We factor the top map as a cofibration followed by an anodyne fibration, and form a pushout
\[\begin{tikzcd}
  A \ar[d,cto] \ar[r,cto] \ar[dr,phantom,"\ulcorner"very near end] & C \ar[d,cto] \ar[r,two heads,"\ano"] & X \ar[d,"\pi"] \\
 B \ar[r,cto]  & D \ar[r] & Y
\end{tikzcd}\]

Finally, we factor that map $D \to Y$ into a cofibration followed by an anodyne fibration
\[\begin{tikzcd}
  A \ar[d,cto] \ar[r,cto] \ar[dr,phantom,"\ulcorner"very near end] & C  \ar[dr,cto] \ar[d,cto] \ar[rr,two heads,"\ano"] &  & X \ar[d,"\pi"] \\
 B \ar[r,cto]  & D \ar[r,cto] & E \ar[ur,dotted] \ar[r,two heads,"\ano"] & Y
\end{tikzcd}\]

By $2$-out-of-$3$, as $\pi$ and the anodyne fibrations are in $\Wcal$, the cofibration $C \cto E$ is also in $\Wcal$, hence the dotted lifting exists by the assumption on $\pi$. It provides a solution to the initial lifting problem.\end{proof}

\begin{remark}\label{rk:last_condition} It hence follows from \cref{lem:acylic_fib_are_inW,lem:W_and_fib_imply_trivial_fib} that a map in $\Ccal$ that has the right lifting property against all cofibrations in $\Wcal$ and all core cofibrations is an anodyne fibrations. However, this argument does not seem to be sufficient to deduce that the last condition of \cref{ass:on_kappa}, which exactly say that it is enough to check these conditions only against cofibrations between $\kappa$-presentable objects for some fixed cardinal $\kappa$.
\end{remark}

\section{Left semi-localizer}
\label{sec:left_semi_localizer}

It does not seem possible to develop an analogue of the theory of localizer for right semi-model categories and weak model categories due to the fact that in a right (or weak) semi-model category only the anodyne fibrations with fibrant target are equivalences, but to determine which objects are fibrant we need to know the class of all equivalences. Concretely, if one tries to write down the type of properties the class of equivalences of a right semi-model category should have we end up with a notion that is not stable under intersection and for which their might not be a class generated by some set of arrows, nor a minimal such class. However, it is possible to develop such a theory for left semi-model categories.

\begin{definition} A left semi-localizer on a structured category $\Ccal$ is a class $\Wcal$ of arrows of $\Ccal$ such that:

\begin{itemize}

\item $\Wcal$ satisfies the $2$-out-of-$3$ condition.

\item All anodyne fibrations are in $\Wcal$.

\item The class of \emph{core} cofibrations that are in $\Wcal$ is stable under transfinite composition and pushout to a cofibrant object.

\end{itemize}

If $S$ is a set of arrows $\Ccal$ we denote by $\Wcal^L(S)$ the smallest left semi-localizer containing $S$. Such left semi-localizer are called small-generated.

\end{definition}

The class of equivalences of a left semi-model structure on $\Ccal$ is a left semi-localizer on $\Ccal$.

Note that we can still prove that a left semi-localizer satisfies the $2$-out-of-$6$ condition as in \cref{lem:2_out_of_6}: the same proof still applies to arrows between cofibrant objects, and we can generalize this to arbitrary object by taking iterative cofibrant replacements.

\begin{theorem}\label{th:left_semi_localizer} If $\Ccal$ is a combinatorial structured category and $\Wcal$ is a class of arrow in $\Ccal$, then $\Wcal$ is the class of equivalence of a combinatorial Fresse left semi-model category if and only if $\Wcal$ is a small-generated left semi-localizer. \end{theorem}

Note that a ``Smith's theorem'' for left semi-model category, i.e. a similar result to the above where small-generation is replaced by the solution set condition has been proved as  Theorem B of \cite{batanin2020Bousfield}, under tractability assumptions.

\begin{proof}[Sketch of proof] The proof is very similar to that of \cref{main_th:Quillen}, with just some cofibrancy hypothesis added in some places and some part that are no longer useful. We will sketch the proof, only insisting on the differences.

As before, we fix $\kappa$ exactly as in \ref{ass:on_kappa} (without the restriction involving $J_0$) and we take $J$ to be the set of all cofibrations between cofibrant objects that are in $\Wcal$ as our set of generating anodyne cofibration.

We first prove that this pre-model structure on $\Ccal$ is core left saturated, this is done in a way very similar to the proof of \cref{lem:left_sat}: the results of section 4 of \cite{henry2019CombWMS} also show\footnote{A careful reader might notice that, contrary to the non-core case, this does not exactly follows from the statement of Theorem~4.1 in \cite{henry2019CombWMS}. It however easily follows from its proof.} that the anodyne cofibrations of the core left saturation of $\Ccal$ are generated by the anodyne cofibrations of $\Ccal$ between $\kappa$-presentable objects and the core acyclic cofibrations between $\kappa$-presentable objects. It is hence enough to show that in $\Ccal$ all core acyclic cofibrations between $\kappa$-presentable objects are anodyne cofibrations, at which point the exact same argument as in the proof of \cref{lem:left_sat} can be used.

Next we define $J_+$ to be the class of arrows in $\Pr_\kappa \Ccal$ that are $\kappa$-small transfinite composition of pushouts of arrows in $J$. It is easy to see, due to the closure property of $J$, that $J \subset J_+$ and an arrow in $J_+$ with cofibrant domain is in $J$. We can immediately deduce that $J_+$ satisfies all the condition of \cref{th:left_semi_ms_from_small} exactly as in the case of Quillen model categories (condition $(iv)$ follows from tractability). This shows that the core left saturation of the pre-model structure on $\Ccal$, is a left semi-model structure. As we have showed that $\Ccal$ is core left saturation, $\Ccal$ it self is indeed a left semi-model category.

We can then apply the same argument as in \cref{th:final_Quillen} to show that the class of equivalences of this left semi-model structure is $\Wcal = \Wcal^L(S)$.
\end{proof}

\begin{remark} It has been showed in \cite{batanin2020Bousfield} (and reproved in \cite{henry2019CombWMS}) that left Bousfield localization of left semi-model category always exists. Hence in the case of left semi-localizer, the discussion we had in \cref{rk:no_Bousfield_loc} simplifies to the fact that the left-semi model structure whose class of equivalences is $\Wcal^L(S)$ is always exactly the left Bousfield localization of the minimal model structure at $S$ (whose equivalences are $\Wcal^L(\emptyset)$).
\end{remark}

\begin{remark}
  An interesting question is whether we can modify \cref{th:left_semi_localizer} to produce a Spitzweck left semi-model structure. A Fresse left semi-model structure which is a pre-model category is a Spitzweck left semi-model structure if it is right saturated, i.e. if every acyclic fibrations is an anodyne fibration. One can of course always take the right saturation (as presented in Section~4 of \cite{henry2019CombWMS}) to obtain a Spitzweck left semi-model category. This means redefining the cofibration as being the map with the left lifting property against acyclic fibrations. While one can consider this a satisfying answer, it has the drawback of changing the underlying structured category and hence not being a perfect analogue of \cref{main_th:Quillen} or \cref{th:left_semi_localizer}.

If one starts with $\Ccal$ a tractable structured category, the model structure obtained is automatically right saturated, so we have a perfect analogue of \cref{main_th:Quillen}: if $\Ccal$ is a tractable combinatorial structured category, then any small-generated left semi-localizer on $\Ccal$ is a the class of equivalences of a tractable combinatorial Spitzweck left semi-model structure.

An analogue of \cref{main_th:Quillen_refined} needs some refinement. One should observe that, in the proof of \cref{th:left_semi_localizer}, if one wants to take a larger set $J$ not confined to core cofibrations the important assumptions are

\begin{itemize}
\item $J$ contains (or generates) all core cofibrations in $\Wcal$ between $\kappa$-presentable objects,
\item any pushout of a an arrow in $J$ which has a cofibrant domain is in $\Wcal$.
\end{itemize}

With these two assumptions, the proof of \cref{th:left_semi_localizer} work unchanged and such a set $J$ can be used as a generating set of anodyne cofibrations of a Fresse left-module structure with $\Wcal$ as sets of weak equivalence. So the question can be rephrased as whether one can find such a set which generates a right saturated pre-model structure. Putting all this together:
\end{remark}

\begin{theorem} let $\Ccal$ be combinatorial structured category. Let $\Wcal$ be left semi-localizer on $\Ccal$. Then $\Wcal$ is the class of equivalences of a combinatorial Spitzweck left semi-model structure if and only if
  \begin{itemize}
  \item $\Wcal$ is a small-generated left semi-localizer.
  \item The (cofibration, anodyne fibration) weak factorization in $\Ccal$ is cofibrantly generated by core cofibrations and by cofibrations $j$ with the property that every pushout of $j$ which has a cofibrant domain is in $\Wcal$.
  \end{itemize}
\end{theorem}

\bibliography{../../../Biblio}{}
\bibliographystyle{plain}

\end{document}